\documentclass[12pt,a4paper]{amsart}
 \usepackage{amssymb,amsthm}
 \usepackage{amsmath,amssymb,amsfonts,amsthm} 
\usepackage{color}                    
\usepackage{subfigure}  
\usepackage{tikz}
\usepackage{graphics}                 
\usepackage{verbatim}   
\usetikzlibrary{arrows,chains,matrix,positioning,scopes}
\makeatletter
\tikzset{join/.code=\tikzset{after node path={%
\ifx\tikzchainprevious\pgfutil@empty\else(\tikzchainprevious)%
edge[every join]#1(\tikzchaincurrent)\fi}}}
\makeatother
\tikzset{>=stealth',every on chain/.append style={join},
         every join/.style={->}}
\tikzstyle{labeled}=[execute at begin node=$\scriptstyle,
   execute at end node=$]
\newtheorem{theorem}{Theorem}[section]
\newtheorem{proposition}[theorem]{Proposition}
\newtheorem{corollary}[theorem]{Corollary}
\newtheorem{lemma}[theorem]{Lemma}

\newtheorem{definition}[theorem]{Definition}

\date{\today}
\title{isomorphisms of quotients of FDD-algebras }

\author{ Saeed Ghasemi}
\address{Department of Mathematics, York University, Toronto, CA}

\begin{document}
\maketitle

\begin{abstract}
We consider isomorphisms between quotient algebras of $\prod_{n=0}^{\infty} \mathbb{M}_{k(n)}(\mathbb{C})$ associated with Borel ideals on $\mathbb{N}$ and prove that it is relatively consistent with \textbf{ZFC} that all of these isomorphisms are trivial, in the sense that they lift to a *-homomorphism from $\prod_{n=0}^{\infty} \mathbb{M}_{k(n)}(\mathbb{C})$ into itself. This generalizes a result of Farah-Shelah who proved this result for centers of these algebras, in its dual form.
\end{abstract}

\section{introduction}
For quotient structures $X/I$, $Y/J$ and $\Phi$ a homomorphism between them, a representation of $\Phi$ is a map $\Phi_{*}: X\rightarrow Y$ such that
\begin{center}
\begin{tikzpicture}
\matrix(m)[matrix of math nodes,
row sep=2.6em, column sep=2.8em,
text height=1.5ex, text depth=0.25ex]
{X&Y\\
X/I&Y/J\\};
\path[->,font=\scriptsize,>=angle 90]
(m-1-1) edge node[auto] {$\Phi_{*}$} (m-1-2)
edge node[auto] {$\pi_{I}$} (m-2-1)
(m-1-2) edge node[auto] {$\pi_{J}$} (m-2-2)
(m-2-1) edge node[auto] {$\Phi$} (m-2-2);
\end{tikzpicture}
\end{center}
commutes, where $\pi_{I}$ and $\pi_{J}$ denote the respective quotient maps. Note that since representation is not required to satisfy any algebraic properties, its existence follows from the axiom of choice. If $\Phi$ has a representation which is a homomorphism itself we say $\Phi$ is \emph{trivial}. The question whether automorphisms between some quotient structures are trivial, sometimes is called \emph{rigidity} question and has been studied for various structures (see for example \cite{ShStep}, \cite{Farah.rig}, \cite{JustOra}, \cite{Rudin}, \cite{PhilWeav} and \cite{FaCalkin}). It turns out that for many structures the answers to these questions highly depend on the set-theoretic axioms.
 To see a brief introduction on rigidity questions for Boolean algebras the reader may refer to \cite[\S 1]{FaSh}. Also for a good reference on C*-algebras refer to \cite{Black}.

By the Gelfand-Naimark duality (see \cite{Black}, \S II.2.2) the rigidity question has an equivalent reformulation in the category of commutative C*-algebras. For non-unital C*-algebra $\mathcal{A}$ the corona of $\mathcal{A}$ is the non-commutative analogue of the \v{C}ech-Stone reminder of a non-compact, locally compact topological space.
Motivated by a question of Brown-Douglas-Fillmore \cite[1.6(ii)]{BDF} the rigidity question for the category of C*-algebras  has been studied for various corona algebras. In particular,  assuming the continuum hypothesis (\textbf{CH}) Phillips and Weaver \cite{PhilWeav} constructed $2^{\aleph_{1}}$ many automorphism of the Calkin algebra over a separable Hilbert space. Since there are only continuum many inner automorphisms of the Calkin algebra this implies that there are many outer automorphisms. On the other hand it was shown by I. Farah \cite{FaCalkin} that under \emph{Todorcevic's Axiom}\footnote{Todorcevic's Axiom is also known as the 'Open Coloring Axiom' and it is a well-known consequence of the 'Proper Forcing Axiom'. } (\textbf{TA}) all automorphisms of the Calkin algebra over a separable Hilbert space are inner and later he showed that the \emph{Proper Forcing Axiom} (\textbf{PFA}) implies that all automorphisms of the Calkin algebra over any Hilbert space are inner \cite{FaAll}.

  In \cite{SamFarah} S. Coskey and I. Farah have conjectured the following:\\

  \textbf{Conjecture 1}: The Continuum Hypothesis implies that the corona of
every separable, non-unital C*-algebra has nontrivial automorphisms.

\textbf{Conjecture 2}: Forcing axioms imply that the corona of every separable,
non-unital C*-algebra has only trivial automorphisms.\\

  In conjecture 2 the notion of triviality refers to a weaker notion than the one used in this paper, and it assures that automorphisms are \emph{definable} in \textbf{ZFC} in a strong sense. In the same article it has been proved that assuming \textbf{CH} every $\sigma$-unital C*-algebra which is either simple or stable has non-trivial automorphisms (see also \cite{FaMcSc}). On the other hand \textbf{TA} and \textbf{MA} imply that all automorphisms of reduced products of UHF-algebras are trivial \cite{Paul}.

In \cite{FaCalkin} the corona algebras of the form  $\prod_{n}\mathbb{M}_{k(n)}(\mathbb{C})/\bigoplus_{n}\mathbb{M}_{k(n)}(\mathbb{C})$ play a crucial role in proving that "\textbf{TA} implies all automorphisms of the Calkin algebra are inner".  In the class of C*-algebras $\prod_{n}\mathbb{M}_{k(n)}(\mathbb{C})$ can be considered as a good counterpart of $P(\mathbb{N})$ in set theory and as it will be clear from next section,  trivial automorphisms of the corona of these algebras give rise to  trivial automorphisms of the boolean algebra $P(\mathbb{N})/\mathcal{F}in$, where $\mathcal{F}in$ is the ideal of all finite subsets of the natural numbers.

 Since the corona of $\prod_{n}\mathbb{M}_{n}(\mathbb{C})$ is \emph{fully countably saturated} structure in the sense of the \emph{ model theory for metric structures} (see \cite{lfms} and \cite{FHS}) and its character density (the smallest cardinality of a dense subset) is the continuum,  under \textbf{CH} it is possible to use a diagonalization argument to show that there are $2^{\aleph_{1}}$ automorphisms of each of these corona algebras. Therefore there are many non-trivial automorphisms (see \cite{FH}, \S2.4).

Given an ideal $\mathcal{J}$ on $\mathbb{N}$ and a sequence of C*-algebras $\{\mathcal{A}_{n}: n\in \mathbb{N}\}$, define the norm-closed ideal
\begin{equation}
\nonumber \bigoplus_{\mathcal{J}}\mathcal{A}_{n}=\{(a_{n})\in \prod_{n}\mathcal{A}_{n}: ~ \lim_{n\rightarrow \mathcal{J}}\|a_{n}\|=0\}
\end{equation}
of $\prod_{n} \mathcal{A}_{n}$, where $\lim_{n\rightarrow \mathcal{J}}\|a_{n}\|=0$ means that for every $\epsilon>0$ the set $\{n\in\mathbb{N}: \|a_{n}\|\geq\epsilon\}\in \mathcal{J}$. The quotient C*-algebra $\prod_{n} \mathcal{A}_{n}/\bigoplus_{\mathcal{J}}\mathcal{A}_{n}$ is usually called the \emph{reduced product} of the sequence $\{\mathcal{A}_{n}: n \in \mathbb{N}\}$ over the ideal $\mathcal{J}$. Clearly if $\mathcal{J}=\mathcal{F}in$ then $\prod_{n}\mathcal{A}_{n}/\bigoplus_{\mathcal{J}}\mathcal{A}_{n}$ is the corona of $\prod_{n}\mathcal{A}_{n}$ and operator algebraists usually call it the \emph{asymptotic sequence algebra} of the sequence $\{\mathcal{A}_{n}: n \in \mathbb{N}\}$.

If ideals $\mathcal{I}$ and $\mathcal{J}$ on $\mathbb{N}$ are Rudin-Keisler isomorphic (see Proposition \ref{main} (1) for the definition) via a bijection $\sigma : \mathbb{N}\setminus A \rightarrow \mathbb{N}\setminus B$ for $A\in\mathcal{I}$ and $B\in\mathcal{J}$, then  for sequences of C*-algebras $\{\mathcal{A}_{n}\}$ and $\{\mathcal{B}_{n}\}$,  an obvious isomorphism $\Phi$ between algebras $\prod_{n}\mathcal{A}_{n}/\bigoplus_{\mathcal{I}} \mathcal{A}_{n}$ and $\prod_{n}\mathcal{B}_{n}/\bigoplus_{\mathcal{J}} \mathcal{B}_{n}$ can be obtained when $\varphi_{n}:\mathcal{A}_{n}\cong \mathcal{B}_{\sigma(n)}$ for every $n\in \mathbb{N}\setminus A$, and $\Phi$ is defined by
\begin{equation}
\nonumber \Phi(\pi_{\mathcal{I}}((a_{n}))) = \pi_{\mathcal{J}}(\varphi_{n}(a_{n})),
\end{equation}
where $\pi_{\mathcal{I}}$ and $\pi_{\mathcal{J}}$ are respective canonical quotient maps. Let us call such an isomorphism \emph{strongly trivial}.

  We will show that if the quotients of $\prod_{n}\mathbb{M}_{n}(\mathbb{C})$ are associated with analytic P-ideals on $\mathbb{N}$ then  it is impossible to construct nontrivial isomorphisms of these algebras without appealing to some additional set-theoretic axioms. This is a consequence of our main result (Theorem \ref{1}) which implies the following corollary.
  \begin{corollary}
  It is relatively consistent with \textbf{ZFC} that for all analytic P-ideals $\mathcal{I}$ and $\mathcal{J}$ on $\mathbb{N}$ all isomorphisms between $\prod_{n}\mathbb{M}_{n}(\mathbb{C})/\bigoplus_{\mathcal{I}}\mathbb{M}_{n}(\mathbb{C})$ and $\prod_{n}\mathbb{M}_{n}(\mathbb{C})/\bigoplus_{\mathcal{J}}\mathbb{M}_{n}(\mathbb{C})$ are strongly trivial. In particular all automorphisms of the corona $\prod_{n}\mathbb{M}_{n}(\mathbb{C})/\bigoplus_{n}\mathbb{M}_{n}(\mathbb{C})$ are strongly trivial.
  \end{corollary}

 It is worth noticing that in general for sequences of separable unital C*-algebras $\mathcal{A}_{n}$ and $\mathcal{B}_{n}$ the question
of whether the algebras $\prod_{n}\mathcal{A}_{n}/\bigoplus_{n} \mathcal{A}_{n}$ and $\prod_{n}\mathcal{B}_{n}/\bigoplus_{n} \mathcal{B}_{n}$ are isomorphic under \textbf{CH} reduces to the
weaker question of whether they are elementary equivalent, when their unit balls
are considered as models for metric structures. This follows from the fact that two
$\kappa$-saturated elementary equivalent structures of character density $\kappa$ are isomorphic,
for any uncountable cardinal $\kappa$ \cite[Proposition 4.13]{FHS}.

    In the main result of this paper we show that assuming there is a measurable cardinal, there is a countable  support iteration of proper and $\omega^{\omega}$-bounding forcings of the form $\mathbb{P}_{\mathcal{I}}$, for a $\sigma$-ideal $\mathcal{I}$, such that in the forcing extension all (isomorphisms) automorphisms of quotients of $\prod_{n}\mathbb{M}_{k(n)}(\mathbb{C})$ over ideals generated by some Borel ideals on $\mathbb{N}$ have continuous representations and if these quotients are associated with analytic P-ideals then all such automorphisms are trivial. This generalizes the main result of  \cite{FaSh} since  the centers of these C*-algebras [see \S2] correspond to the Boolean algebras handled in \cite{FaSh}. The assumption that there exists a measurable cardinal is there merely to make sure that $\Pi^{1}_{2}$ sets in the generic extension have Baire-measurable uniformizations.
 We  use a slight modification of the Silver forcing instead of the creature forcing used in \cite{FaSh}. In section 3 we give a brief introduction to some the properties of these forcings and their countable support iterations.
 As in \cite{FaSh} the results of this paper are consistent with the Calkin algebra having an outer automorphism [corollary \ref{654}].

 We follow \cite{FaCalkin} and use the terminology 'FDD-algebras' (Finite Dimensional Decomposition) for spatial representations of $\prod_{n}\mathbb{M}_{k(n)}(\mathbb{C})$ on separable Hilbert spaces, but  throughout this paper we usually identify FDD-algebras with $\prod_{n}\mathbb{M}_{k(n)}(\mathbb{C})$ for some sequence of natural numbers $\{k(n)\}$.\\

\textbf{ACKNOWLEDGMENTS}. I am indebted to my supervisor Ilijas Farah for illuminative suggestions and supervision. I would like to thank Marcin Sabok for pointing out that in lemma \ref{999} any large cardinal assumption can be removed. I would also like to thank the anonymous referee for making number of useful suggestions.

\section{ fdd-algebras and closed ideals associated with  borel ideals}
For a  separable infinite dimensional Hilbert space $H$  let $\mathcal{B}(H)$ denote the space of all bounded linear operators on $H$. For a C*-algebra $A$ we use $A_{\leq 1}$ to denote the unit ball of $A$.

 \begin{definition}\label{129}
  Fix a separable infinite dimensional Hilbert space $H$ with an orthonormal basis $\{e_{n}:n\in\mathbb{N}\}$.
Let $\vec{E} = (E_n)$ be a partition of $\mathbb{N}$ into finite intervals, i.e., a finite set of consecutive natural numbers,  and $\mathcal{D}[\vec{E}]$ denote the von Neumann algebra of all operators in $\mathcal{B}(H)$ such that the subspace spanned by $\{e_{i} : i\in E_{n}\}$ is invariant. These algebras are called FDD-algebras.
\end{definition}
 Clearly $\mathcal{D}[\vec{E}]$ is isomorphic to $\prod_{n=0}^{\infty}\mathbb{M}_{|E_{n}|}(\mathbb{C})$.
 The unit ball of $\mathcal{D}[\vec{E}]$ is a Polish space when equipped with the strong operator topology and this allows us to use tools from descriptive set theory in this context.

 For $M\subseteq \mathbb{N}$ let $P_{M}^{\vec{E}}$ be the projection on the closed span of $\bigcup_{n\in M}  \{e_{i} : i \in \vec{E}_{n}\}$  and $\mathcal{D}_{M}[\vec{E}]$ be the closed ideal $P_{M}^{\vec{E}}\mathcal{D}_{M}[\vec{E}] P_{M}^{\vec{E}} = P_{M}^{\vec{E}}\mathcal{D}[\vec{E}]$. For a fixed $\vec{E}$ we often drop the superscript and write $P_{M}$ and $P_{n}$ instead of  $P_{M}^{\vec{E}}$ and $P_{\{n\}}^{\vec{E}}$.

For a  Borel ideal $\mathcal{J}$ on $\mathbb{N}$, the subspace $\mathcal{D}^{\mathcal{J}}[\vec{E}]=\overline{\bigcup_{X\in \mathcal{J}}\mathcal{D}_{X}[\vec{E}]}$
is a closed ideal of $\mathcal{D}[\vec{E}]$. Equivalently
\begin{equation}
\nonumber \mathcal{D}^{\mathcal{J}}[\vec{E}]=\{(a_{n})\in \mathcal{D}[\vec{E}]: ~ \lim_{n\rightarrow \mathcal{J}}\|a_{n}\|=0\}.
\end{equation}

   Let $\mathcal{C}^{\mathcal{J}}[\vec{E}]=\mathcal{D}[\vec{E}]/\mathcal{D}^{\mathcal{J}}[\vec{E}]$ and $\pi_{\mathcal{J}}$ be the natural quotient map. For operators $a$ and $b$ in $\mathcal{D}[\vec{E}]$ we usually write $a=^{\mathcal{J}}b$ instead of $a-b\in \mathcal{D}^{\mathcal{J}}[\vec{E}]$.

An ideal $\mathcal{J}$ on $\mathbb{N}$ is a \emph{P-ideal} if for every sequence $\{A_{n}\}$ of sets in $\mathcal{J}$ there exists $A\in\mathcal{J}$ such that $A_{n}\setminus A$ is finite, for every $n$.

The following theorem is the main result of this paper.
\begin{theorem}\label{1}
Assume there is a measurable cardinal. There is a forcing extension in which  for partitions $\vec{E}$ and $\vec{F}$   of the natural numbers into finite intervals, if  $\mathcal{I}$ and $\mathcal{J}$ are Borel ideals on the natural numbers, then the following are true.
\begin{enumerate}
\item Any automorphism $\Phi : \mathcal{C}^{\mathcal{J}}[\vec{E}]\rightarrow \mathcal{C}^{\mathcal{J}}[\vec{E}]$ has a (strongly) continuous  representation.
    \item  Any isomorphism $\Phi : \mathcal{C}^{\mathcal{I}}[\vec{E}]\rightarrow \mathcal{C}^{\mathcal{J}}[\vec{F}]$ has a  continuous representation.\\

        If $\mathcal{I}$ and $\mathcal{J}$ are analytic P-ideals then\\

\item Any automorphism $\Phi : \mathcal{C}^{\mathcal{J}}[\vec{E}]\rightarrow \mathcal{C}^{\mathcal{J}}[\vec{E}]$ has a *-homomorphism representation.
    \item  Any isomorphism $\Phi : \mathcal{C}^{\mathcal{I}}[\vec{E}]\rightarrow \mathcal{C}^{\mathcal{J}}[\vec{F}]$ has a *-homomorphism representation.\\
\end{enumerate}
\end{theorem}
The following corollary follows from the proof of theorem \ref{1} and does not require any large cardinal assumption. See \S5 for definition of local triviality.
\begin{corollary}\label{local}
There is a forcing extension in which if $\mathcal{I}$ and $\mathcal{J}$ are (P)-ideals on $\mathbb{N}$,  any *-homomorphism $\Phi : \mathcal{C}^{\mathcal{I}}[\vec{E}]\rightarrow \mathcal{C}^{\mathcal{J}}[\vec{F}]$ has a locally (*-homomorphism) continuous representation.
\end{corollary}
In order to avoid making notations more complicated we only prove this theorem for automorphisms and it is easy to see that the same proof works for isomorphisms.

  In our forcing extension every such isomorphism has a simple description as it turns out that these isomorphisms are implemented by isometries between "co-small" subspaces. For partitions $\vec{E}=(E_{n})$ and  $\vec{F}=(F_{n})$ of $\mathbb{N}$ into finite intervals in the following proposition let  $\mathcal{D}[\vec{E}]$ and $\mathcal{D}[\vec{F}]$ be the FDD-algebras associated with $\vec{E}$ and $\vec{F}$ with respect to fixed orthonormal basis $\{e_{n}: n\in \mathbb{N}\}$ and $\{f_{n}: n\in\mathbb{N}\}$ for Hilbert spaces $H$ and $K$ respectively. Also let
   \begin{eqnarray}
 \nonumber H_{n}&=&{span\{e_{i}: i\in E_{n}\}} \qquad P_{n}=Proj(H_{n})\\
  \nonumber K_{n}&=&{span\{f_{i}: i\in F_{n}\}} \qquad Q_{n}=Proj(K_{n}).
 \end{eqnarray}

  \begin{proposition}\label{main}
Assume there is a measurable cardinal. There is a forcing extension in which the following holds. Assume $\mathcal{I}$, $\mathcal{J}$ are analytic P-ideals on $\mathbb{N}$ and $\vec{E}=(E_{n})$, $\vec{F}=(F_{n})$ are partitions of $\mathbb{N}$ into finite intervals.  Then there is an isomorphism $\Phi: \mathcal{C}^{\mathcal{I}}[\vec{E}]\mapsto \mathcal{C}^{\mathcal{J}}[\vec{F}]$ if and only if
\begin{enumerate}
\item $\mathcal{I}$ and $\mathcal{J}$ are Rudin-Keisler isomorphic, i.e.,
there are sets $B\in \mathcal{I}$ and $C \in \mathcal{J}$ and a bijection $\sigma: \mathbb{N}\setminus B\mapsto \mathbb{N}\setminus C$ such that $X\in \mathcal{I}$ if and only if $\sigma[X]\in \mathcal{J}$, and
\item $|E_{n}|=|F_{\sigma(n)}|$ for every $n\in \mathbb{N}\setminus B$.
\end{enumerate}

    Moreover, for every $n\in \mathbb{N}\setminus B$  there is a linear isometry $u_{n}: H_{n}\mapsto K_{\sigma(n)}$ such that if $u=\sum_{n\in\mathbb{N}\setminus B}u_{n}$ , then the map $a\mapsto uau^{*}$ is a representation of $\Phi$.
 \end{proposition}

  \begin{proof}
The inverse direction of the first statement is trivial. To prove the forward direction assume $\Phi: \mathcal{C}^{\mathcal{I}}[\vec{E}]\mapsto \mathcal{C}^{\mathcal{J}}[\vec{F}]$ is an isomorphism. Using theorem \ref{1} there is a forcing extension in which there is a *-homomorphism $\Psi: \mathcal{D}[\vec{E}]\mapsto \mathcal{D}[\vec{F}]$ which is a representation of $\Phi$. For every $n$ we have $\Psi(P_{n})(K)\subseteq Q_{m}(K)$ for some $m$. It is easy to see that since $\Phi$ is an isomorphism there are $B\in\mathcal{I}$, $C\in\mathcal{J}$ and a bijection $\sigma:\mathbb{N}\setminus B\mapsto \mathbb{N}\setminus C$ such that for every $n\in \mathbb{N}\setminus B$  we have $\Psi(P_{n})(K)= Q_{\sigma(n)}(K)$. The map $\sigma$ witnesses that $\mathcal{I}$ and $\mathcal{J}$ are Rudin-Keisler isomorphic.  Moreover, for every one-dimensional projection $P\in \mathcal{B}(H_{n})$ the image, $\Psi(P)$,  is also a one-dimensional projection in $B(K_{\sigma(n)})$. In particular $|E_{n}|=|F_{\sigma(n)}|$.

 Now for every $n\in\mathbb{N}\setminus B$ assume $E_{n}=[k_{n},k_{n+1}]$ and define a unitary $a\in B(H_{n})$ by
 \begin{equation*}
    a(e_{k_{i}}) = \begin{cases}
               e_{k_{i}+1}               &  k_{n}\leq i<k_{n+1}\\
               e_{k_{n}}               & i=k_{n+1} .
           \end{cases}
\end{equation*}

Fix $\xi_{0}\in K_{\sigma(n)}$. Let $b=\Psi(a)$ and $\xi_{j}=b^{j}(\xi_{0})$ ($b^{j}$ is the $j$-th power of $b$) for each $0\leq j < |E_{n}|$. Then $\{\xi_{j}: 0\leq j< |E_{n}|\}$ forms a basis for $K_{\sigma(n)}$ and $e_{k_{j}}\mapsto \xi_{j}$ defines an isometry $u_{n}$ as required.\\
 Now $u=\bigoplus_{n\in \mathbb{N}\setminus B}u_{n}$ is an isometry from $\bigoplus_{n\in \mathbb{N}\setminus B}H_{n}$ to $\bigoplus_{n\in \mathbb{N}\setminus C} K_{n}$ such that $\Psi(a)- uau^{*}\in \mathcal{D}^{\mathcal{J}}(\vec{F})$ for all $a\in \mathcal{D}[\vec{E}]$.
 \end{proof}

 As we mentioned in the introduction the result of Farah and Shelah \cite{FaSh} can be obtained from theorem $\ref{1}$.
 \begin{corollary}
If there is measurable cardinal,  there is a forcing extension in which
 every isomorphism between quotient Boolean algebras $P(\mathbb{N})/\mathcal{I}$ and
$P(\mathbb{N})/\mathcal{J}$ over Borel ideals has a continuous representation.
 \end{corollary}
 \begin{proof}
   Let $E_{n}=\{n\}$. Then $\mathcal{D}[\vec{E}]\cong \ell_{\infty}$ is the standard atomic masa (maximal abelian subalgebra) of $B(H)$ and for
\begin{equation}
\nonumber \hat{\mathcal{J}}=\{(\alpha_{n})\in \ell_{\infty} : \lim_{n\rightarrow\mathcal{J}} \alpha_{n}=0 \}
\end{equation}
  clearly $\mathcal{C}^{\mathcal{J}}[\vec{E}]= \ell_{\infty}/\hat{\mathcal{J}}\cong C(st(P(\mathbb{N})/\mathcal{J}))$ where $st(P(\mathbb{N})/\mathcal{J})$ is the Stone space of $P(\mathbb{N})/\mathcal{J}$. The duality between categories implies that
 every isomorphism $\Phi$ between $P(\mathbb{N})/\mathcal{I}$ and $P(\mathbb{N})/\mathcal{J}$ corresponds to an isomorphism $\tilde{\Phi}$ between $C(st(P(\mathbb{N})/\mathcal{I}))$ and $C(st(P(\mathbb{N})/\mathcal{J}))$. The continuous map witnessing the topological triviality of $\tilde{\Phi}$ corresponds to a continuous map witnessing the topological triviality of $\Phi$.
 \end{proof}
 For any partition $\vec{E}$ let  $Z(\mathcal{C}^{\mathcal{J}}[\vec{E}])$ denote the center of $\mathcal{C}^{\mathcal{J}}[\vec{E}]$ and $U(n)$ be the compact group of all unitary $n\times n$ matrices equipped with the bi-invariant normalized Haar measure $\mu$. More generally the following are true.
\begin{lemma}\label{987}
For any ideal $\mathcal{J}$
\begin{equation}
\nonumber Z(\mathcal{C}^{\mathcal{J}}[\vec{E}])= \frac{Z(\mathcal{D}[\vec{E}])} {\mathcal{D}^{\mathcal{J}}[\vec{E}]\cap Z(\mathcal{D}[\vec{E}])}.
\end{equation}
\end{lemma}
\begin{proof}
Clearly we have $Z(\mathcal{D}[\vec{E}])/ (\mathcal{D}^{\mathcal{J}}[\vec{E}]\cap Z(\mathcal{D}[\vec{E}]))\subseteq Z(\mathcal{C}^{\mathcal{J}}[\vec{E}])$. For the other direction it is enough to show that for every $a+ \mathcal{D}^{\mathcal{J}}[\vec{E}]\in Z(\mathcal{C}^{\mathcal{J}}[\vec{E}])$ there exists a $a^{\prime}\in Z(\mathcal{D}[\vec{E}])$ such that $a-a^{\prime}\in \mathcal{D}^{\mathcal{J}}[\vec{E}] $, in other words every element of $Z(\mathcal{C}^{\mathcal{J}}[\vec{E}])$ can be lifted to an element of $Z(\mathcal{D}[\vec{E}])$. Let $a=(a_{n})$ be such that each $a_{n}$ belongs to $M_{|E_{n}|}(\mathbb{C})$ and $a+ \mathcal{D}^{\mathcal{J}}[\vec{E}]\in Z(\mathcal{C}^{\mathcal{J}}[\vec{E})]$. For every $n$ let
\begin{equation}
\nonumber a^{\prime}_{n}=\int_{u\in U(|E_{n}|)} ua_{n}u^{*} d\mu
\end{equation}
and since $\mu$ is bi-invariant, for every unitary $u\in M_{|E_{n}|}(\mathbb{C})$ we have $ua^{\prime}_{n}u^{*}=a^{\prime}_{n}$.
If $a^{\prime}=(a^{\prime}_{n})$ then  $a^{\prime}\in Z(\mathcal{D}[\vec{E}])$ and $a-a^{\prime}\in \mathcal{D}^{\mathcal{J}}[\vec{E}]$.
\end{proof}

\begin{proposition}
 $Z(\mathcal{C}^{\mathcal{J}}[\vec{E}])\cong C(st(P(\mathbb{N})/\mathcal{J}))$.
\end{proposition}
\begin{proof}
Clearly we have $Z(\mathcal{D}[\vec{E}])\cong \ell_{\infty}$ and $\mathcal{D}^{\mathcal{J}}[\vec{E}]\cap Z(\mathcal{D}[\vec{E}])\cong \hat{\mathcal{J}}$.
Therefore by lemma \ref{987} we have  $Z(\mathcal{C}^{\mathcal{J}}[\vec{E}])\cong \ell_{\infty}/\hat{\mathcal{J}}\cong C(st(P(\mathbb{N})/\mathcal{J}))$.
\end{proof}


\section{groupwise silver forcing and forcings of the form $\mathbb{P}_{\mathcal{I}}$}

In this section we introduce the forcing used in this context and provide some preliminaries on the properties of these forcings which are used throughout this paper. To see more on forcing and these properties the reader may refer to \cite{Bart} and \cite{ShProper}.

 A forcing notion $\mathbb{P}$ is called to be $Suslin$ if its underlying set is an
analytic set of reals and both $\leq$ and $\perp$ are analytic relations.

The following is similar to infinitely equal forcing EE \cite[\S 7.4.C]{Bart}.
Let $\vec{I} = (I_n)$ be a partition of $\mathbb{N}$ into non-empty finite intervals and $G_{n}$ be a finite set, for each $n\in\mathbb{N}$. We denote the set of the \emph{reals} by $\mathbb{R}=\prod_{n}G_{n}$ endowed with the product topology. For each $n$ define $F_{n}^{\vec{I}}=\prod_{i\in I_{n}}G_{i}$ and let $F^{\vec{I}}=\prod_{n\in\mathbb{N}}F^{\vec{I}}_{n}$.
 Moreover for any $X\subseteq \mathbb{N}$ let $F_{X}^{\vec{I}}=\prod_{n\in X}F^{\vec{I}}_{n}$. In particular if $X$ is an interval such as $\{k, k+1,\dots, \ell\}$ we use $F_{[k,\ell]}$ to denote $F_{\{k, k+1, \dots, \ell\}}$. For a fixed partition $\vec{I}$ we sometimes drop the superscript $\vec{I}$.

 Fix  a partition $\vec{I}=(I_{n})$  of the natural numbers into finite intervals.
Define  the groupwise Silver forcing $\mathbb{S}_{F^{\vec{I}}}$  associated with $F^{\vec{I}}$, to be the following forcing notion: A condition $p\in \mathbb{S}_{F^{\vec{I}}}$ is a function from $M\subseteq \mathbb{N}$ into $\bigcup_{n=0}^{\infty} F_{n}^{\vec{I}}$, such that $\mathbb{N} \setminus M$ is infinite and $p(n)\in F_{n}^{\vec{I}}$.  A condition $p$ is stronger than $q$ if $p$ extends $q$. Each condition $p$ can be identified with $[p]$, the set of all its extensions to $\mathbb{N}$, as a compact subset of $F^{\vec{I}}$.  For a generic $G$, $f = \bigcup \{p: p\in G\}$ is the generic real.


 Recall that a forcing notion $\mathbb{P}$ is $\omega^{\omega}$-\emph{bounding} if for every $p\in \mathbb{P}$ and a $\mathbb{P}$-name for a function $\dot{f}: \omega\rightarrow \omega$ there are $q\leq p$ and $g\in \omega^{\omega}\cap V$ such that $q\Vdash \dot{f}(\check{n})\leq \check{g}(\check{n}) ~~\forall n$.
 \begin{theorem}
$\mathbb{S}_{F^{\vec{I}}}$ is a  proper and $\omega^{\omega}$-bounding forcing.
\end{theorem}

\begin{proof}
Let $\mathcal{M}\prec H_{\theta}$ for a large enough $\theta$, be a countable transitive model of \textbf{ZFC} containing  $\vec{I}$ and $\mathbb{S}_{F^{\vec{I}}}$. Suppose $\{A_{n}~:~n\in\mathbb{N}\}$ is the set of all maximal antichains in $\mathcal{M}$ and $q\in \mathbb{S}_{F^{\vec{I}}}$ is given. First we claim that there exists $p\in\mathbb{S}_{F^{\vec{I}}}$ such that for every $n$ the set $\{q\in A_{n} ~:~ q~is~compatible ~with~p\}$ is finite. To see this let $p\leq_n q$ if and only if $q\subset p$ and the first $n$ elements that are not in the domain of $q$ are not in the domain of $p$.
We build a fusion sequence $ p_{0}\geq_{0} p_{1}\geq_{1} \dots p_{n}\geq_{n} p_{n+1}\geq_{n+1} \dots$ recursively. For the given $q$ let $p_{0}=q$ and suppose $p_{n}$ is chosen. Let $B= \{k_{1} \dots k_{n}\}$ be the set of first $n$ elements of $\mathbb{N} \setminus dom(p_n)$ ordered increasingly.  Since $A_{n}$ is a maximal antichain, $p_{n}$ is compatible with some $s\in A_{n}$. Let $p_{n+1}= p_{n}\cup s\upharpoonright_{(k_{n},\infty)}$. Note that $p_{n+1}$ is compatible with only finitely many elements of $A_{n}$, namely, only possibly those elements $t\in A_{n}$ which $t(i)\neq s(i)$ for some $i\in [0, k_{n})$. Let $p=\bigcup_{n} p_{n}$ be the fusion of the above sequence. For every $n$ the set $C_{n}= \{q\in A_{n} ~:~ q~is~compatible ~with~p\}$ is finite and predense below $p$ for every $n$. Therefore $A_{n}\cap \mathcal{M}$ contains $C_{n}$ and is predense below $p$.

To see $\mathbb{S}_{F^{\vec{I}}}$ is $\omega^{\omega}$-bounding assume $\dot{f}$ is an $\mathbb{S}_{F^{\vec{I}}}$-name such that $q\Vdash \dot{f} : \mathbb{N} \rightarrow\mathbb{N}$. As above we build a fusion sequence $ q=p_{0}\geq_{0} p_{1}\geq_{1} \dots p_{n}\geq_{n} p_{n+1}\geq_{n+1} \dots$. Let $B$ be defined as above and $\{r_{j} : j< k\}$ be the list of all functions $r : B \rightarrow \bigcup_{i\in C} F_{i}^{\vec{I}}$ where $C$ is a finite set such that for every $1\leq j \leq n$ we have $r(k_{j})\in F_{i}^{\vec{I}}$ for some $i\in C$. Successively find $p_n = p_{n}^{0}\geq_{n} p_{n}^{1}\geq_{n} \dots \geq_{n}p_{n}^{k-1} $ such that:
\begin{equation}
\nonumber p_{n}^{j} \cup r_{j}\Vdash \dot{f}(n) = \check{a}_{n}^j.
\end{equation}

Let $p_{n+1}=\bigcup p_{n}^{j}$ and $D_n= \{a_{n}^j : j<k \}$.  Now the fusion of this sequence $p$ forces that for every $n$ we have $f(n)\in D_{n}$. Define a ground model map $g : \mathbb{N} \rightarrow\mathbb{N}$ by $g(n)$ to be the largest element of $D_{n}$. Therefore $p$ forces that $g(n)\geq f(n)$ for all $n$.

\end{proof}

 The following property is the main reason that we use $\mathbb{S}_{F^{\vec{I}}}$ in this context.
\begin{definition}\label{125}
 We say a forcing notion $\mathbb{P}$ captures $F^{\vec{I}}$ if there exists a $\mathbb{P}$-name for a real $\dot{x}$  such that for every $p\in\mathbb{P}$ there is an infinite $M\subseteq \mathbb{N}$ such that for every $a\in F^{\vec{I}}_{M}$ there is $q_{a}\leq p$ such that $q_{a}\Vdash \dot{x}\upharpoonright_{ M}  = \check{a}$.
\end{definition}

\begin{lemma}
For any partition $\vec{I}$ of $\mathbb{N}$ into finite intervals, $\mathbb{S}_{F^{\vec{I}}}$ captures $F^{\vec{I}}$.
\end{lemma}
\begin{proof}
Suppose $\dot{x}$ is the canonical $\mathbb{S}_{F^{\vec{I}}}$-name for the generic real and $p\in\mathbb{S}_{F^{\vec{I}}}$ is given. Let $M$ be an infinite subset of $\mathbb{N}\setminus dom(p)$ such that $\mathbb{N}\setminus (M\cup dom(p))$ is also infinite. For every $a\in F^{\vec{I}}_M$ let $q_{a}= p \cup a$. Since $\mathbb{N}\setminus (M\cup dom(p))$ is infinite, $q_{a}$ is a condition in $\mathbb{S}_{F^{\vec{I}}}$ and  $q_{a}\Vdash \dot{x}\upharpoonright_{ M} = \check{a}$.
\end{proof}
 \begin{lemma}\label{100}
 For every $\mathbb{S}_{F^{\vec{I}}}$-name  for a real $\dot{x}$ and $q\in \mathbb{S}_{F^{\vec{I}}}$ there are $p\leq q$ and a continuous function $f:p\rightarrow \mathbb{R}$ such that $p$ forces $f(\dot{r}_{gen})= \dot{x}$, where $\dot{r}_{gen}$ is the canonical name for the generic real.
\end{lemma}
\begin{proof}
Assume $q$ forces that $\dot{x}$ is a  $\mathbb{S}_{F^{\vec{I}}}$-name for a real.  By identifying each condition with the corresponding compact set we can find a fusion sequence $\{p_{s} :s\in  \bigcup_{n}F_{[0,n)}^{\vec{I}}\}$ such that for each $s\in F_{[0,n)}^{\vec{I}}$  (here $s$ just would be used as an index) $p_{s}\Vdash \dot{x}\upharpoonright_{[0,n)} = u_{s}$ for some $u_{s}\in F_{[0,n)}^{\vec{I}} $. Let
\begin{equation}
\nonumber p=\bigcap_{n\in\mathbb{N}}\bigcup_{s\in  F_{[0,n)}^{\vec{I}}} p_{s}
\end{equation}
 be the fusion. For each $y\in p$ let $b\in F^{\vec{I}}$ be the branch such that $y\in p_{b\upharpoonright _{[0,n)}}$ for each $n$. Define $f(y)\upharpoonright_{[0,n)}= u_{b\upharpoonright _{[0,n)}}$. $f$ is a continuous map and $y\in p_{b\upharpoonright _C}$ implies $ d(f(y), \dot{x})< 2^{-n}$.
 Therefore $p\Vdash f(\dot{r}_{gen})= \dot{x}$.
\end{proof}

The above lemma shows that $\mathbb{S}_{F^{\vec{I}}}$ satisfies the ,so called, \emph{continuous reading of names}. This can also be seen by noticing that
the groupwise Silver forcing can be viewed as a forcing with Borel $\mathcal{I}$-positive sets, $\mathbb{P}_{\mathcal{I}}= \mathcal{B}(\mathbb{R})/\mathcal{I}$, for a $\sigma$-ideal $\mathcal{I}$, where $\mathcal{I}$ is the $\sigma$-ideal $\sigma$-generated by partial functions with cofinite domains. These forcings are studied by J. Zapletal in \cite{ZapDes} and \cite{ZapId}.  Since groupwise Silver forcings (as well as the random forcing, which will be used in our iteration) are proper and conditions are compact sets, by a theorem of Zapletal \cite[Lemma 2.2.1 and Lemma 2.2.3]{ZapDes} the continuous reading of names is equivalent to the forcing being $\omega^{\omega}$-bounding.

\begin{lemma}[J. Zapletal]\label{zap}
Let $\mathcal{I}$ be a $\sigma$-ideal on a Polish space $X$ and $\mathbb{P}_{\mathcal{I}}$ is a proper forcing. Then following are equivalent.
\begin{enumerate}
\item $\mathbb{P}_{\mathcal{I}}$ is $\omega^{\omega}$-bounding.
\item Compact sets are dense in $\mathbb{P}_{\mathcal{I}}$ and $\mathbb{P}_{\mathcal{I}}$ has continuous reading of names.
\end{enumerate}
\end{lemma}
Note that it is essential that compact conditions are dense in the poset, since the Cohen forcing is proper and of the form $\mathbb{P}_{\mathcal{I}}$, but it is not $\omega^{\omega}$-bounding, yet it does have the continuous reading of names.

In Zapletal's theory the countable support iteration of forcings of the form $\mathbb{P}_{\mathcal{I}}$ has been studied for reasonably definable ideals called $iterable$ \cite[Definition 3.1.1]{ZapDes}. The following is a generalization of the classical \emph{Fubini product} of two ideals.
\begin{definition}
[J. Zapletal] For a countable ordinal $\alpha$ and $\sigma$-ideals $\{\mathcal{I}_{\xi}: \xi\in\alpha\}$ on the reals, the Fubini product, $\prod_{\xi\in\alpha}\mathcal{I}_{\xi}$, is the ideal
 on $\mathbb{R}^{\alpha}$ defined as the collection of all sets $A\subseteq \mathbb{R}^{\alpha}$ for which the player I has a winning strategy in the game $G(A)$ as follows: at stage $\beta\in \alpha$ player I plays a set $B_{\beta}\in \mathcal{I}_{\beta}$ and player II produces a real $r_{\beta}\in \mathbb{R}\setminus B_{\beta}$. Player II wins the game $G(A)$ if the sequence $\{r_{\beta}: \beta\in\alpha\}$ belongs to the set $A$.
\end{definition}
It is easy to see that $\prod_{\xi\in\alpha}\mathcal{I}_{\xi}$ is a $\sigma$-ideal on $\mathbb{R}^{\alpha}$ since player I can always combine countably many of his winning strategies into one. In the presence of large cardinals the game $G(A)$ is always determined for iterable ideals. However without large cardinals we need some additional definability assumptions on the ideals to guarantee that the game $G(A)$ is determined, see \cite{ZapDes} section 3.3.

Recall that for Polish spaces $X,Y$ and $A\subseteq X\times Y$, for any $x\in X$ the vertical section of $A$ at $x$ is the set $A_{x}=\{y\in Y: (x,y)\in A\}$.
\begin{definition}
A $\sigma$-ideal $\mathcal{I}$ on a polish space $X$ is $\Pi_{1}^{1}$ on $\Sigma_{1}^{1}$ if for every $\Sigma_{1}^{1}$ set $B\subseteq 2^{\mathbb{N}}\times X$ the set $\{x\in 2^{\mathbb{N}}: B_{x}\in \mathcal{I}\}$ is $\Pi_{1}^{1}$.
\end{definition}
 Some commonly used ideals fail to be $\Pi_{1}^{1}$ on $\Sigma_{1}^{1}$, e.g. a $\sigma$-ideal $\mathcal{I}$ for which the forcing $\mathbb{P}_{\mathcal{I}}$ is proper and adds a dominating real is not $\Pi_{1}^{1}$ on $\Sigma_{1}^{1}$. However the $\sigma$-ideals corresponding to the Silver forcing, random forcing and many other natural forcings are $\Pi_{1}^{1}$ on $\Sigma_{1}^{1}$ . In fact
 for a $\sigma$-ideal $\mathcal{I}$ on $\mathbb{R}$ if the poset $\mathbb{P}_{\mathcal{I}}$ consists of compact sets and is Suslin, proper and $\omega^{\omega}$-bounding, then $\mathcal{I}$ is  $\Pi_{1}^{1}$ on $\Sigma_{1}^{1}$ (see \cite{ZapDes}, Appendix C).

The following is due to V. Kanovei and J. Zapletal.
\begin{theorem}\label{ZapKan}
Suppose $\alpha$ is a countable ordinal and $\{\mathcal{I}_{\xi}: \xi < \alpha\}$ is a sequence of $\Pi_{1}^{1}$ on $\Sigma_{1}^{1}$ $\sigma$-ideals on the reals. Then the poset $\mathcal{B}(\mathbb{R}^{\alpha})/\prod_{\xi\in\alpha}\mathcal{I}_{\xi}$ is forcing equivalent to the countable support iteration of the ground model forcings  $\{\mathbb{P}_{\mathcal{I}_{\xi}},~\xi\leq\alpha\}$ of length $\alpha$.
\end{theorem}
\begin{proof}
The proof is similar to \cite[Lemma 3.3.1, corollary 3.3.2]{ZapDes}, where it is stated for the case $\mathcal{I}_{\xi}= \mathcal{I}$ for all $\xi<\alpha$.
\end{proof}

 We will occasionally  use the following property of countable support forcing iterations, which was defined in \cite{FaSh}, to prove our main theorem.
\begin{definition}
 Assume $\mathbb{P}_{\kappa}=\{\mathbb{P}_{\xi}, \dot{\mathbb{Q}}_{\eta}:\xi\leq \kappa , \eta< \kappa\}$ is a countable support forcing iteration such that each $\dot{\mathbb{Q}}_{\eta}$ is a $\mathbb{P}_{\eta}$-name for a ground-model forcing notion which adds a generic real $\dot{g}_{\eta}$. We say $\mathbb{P}_{\kappa}$ has continuous reading of names if for every $\mathbb{P}_{\kappa}$-name $\dot{x}$ for a new real, the set of conditions $p$ such that there exist a countable $S\subset \kappa$, a compact $K\subset \mathbb{R}^{S}$, and a continuous $h: K\rightarrow \mathbb{R}$ such that
\begin{equation}
\nonumber p\Vdash "\langle\dot{g}_{\xi}: ~\xi\in S\rangle\in K ~and~ \dot{x}= h(\langle\dot{g}_{\xi} :~ \xi\in S\rangle)"
\end{equation}
is dense.
\end{definition}
\begin{proposition}\label{111}
If $\mathbb{P}_{\kappa}=\{\mathbb{P}_{\xi}, \dot{\mathbb{Q}}_{\eta}:\xi\leq \kappa , \eta< \kappa\}$ is a countable support iteration of ground model forcings such that each $\dot{\mathbb{Q}}_{\eta}$ is a $\mathbb{P}_{\eta}$-name for a Suslin, proper and $\omega^{\omega}$-bounding partial order of the form $\mathbb{P}_{\mathcal{I}}$ such that compact conditions form a dense subset, then $\mathbb{P}_{\kappa}$ has the continuous reading of names.
\end{proposition}
\begin{proof}
Suppose $\mathbb{Q}_{\xi}= \mathbb{P}_{\mathcal{I}_{\xi}}$ and $p\in \mathbb{P}_{\kappa}$ forces that $\dot{x}$ a $\mathbb{P}_{\kappa}$-name for a real. Let $S$ be the support of $p$ and $\mathcal{I}=\prod_{\xi\in S} \mathcal{I}_{\xi}$. If $\mathbb{P}_{\mathcal{I}}= \mathcal{B}(\mathbb{R}^{S})/ \mathcal{I}$, by theorem \ref{ZapKan} we can assume that $\dot{x}$ is a $\mathbb{P}_{\mathcal{I}}$-name and since $\mathbb{P}_{\mathcal{I}}$ is proper, $\omega^{\omega}$-bounding and compact conditions form a dense subset, by Zapletal's characterization of continuous reading of names for these posets, lemma \ref{zap},  there are a compact condition $q\leq p$ and a continuonus function $h: q\mapsto \mathbb{R}$ such that $q\Vdash   h(\langle\dot{g}_{\xi} :~ \xi\in S\rangle)=\dot{x}$.
\end{proof}

Let $\mathcal{I}$ be a $\sigma$-ideal and $\mathcal{M}$ be an elementary submodel of some large enough structure containing $\mathcal{I}$.  A real $x$ is called $\mathcal{M}$-generic if the set $\{B\in \mathbb{P}_{\mathcal{I}}\cap \mathcal{M}: x\in B\}$ is an $\mathcal{M}$-generic filter on $\mathbb{P}_{\mathcal{I}}$.  The poset $\mathbb{P}_{\mathcal{I}}$ is proper if and only if for every such $\mathcal{M}$ and every $\mathcal{I}$-positive set $B\in \mathbb{P}_{\mathcal{I}}\cap\mathcal{M}$ the set $\{x\in B : \text{$x$ is $\mathcal{M}$-generic}\}$ is $\mathcal{I}$-positive \cite[Lemma 2.1.2]{ZapDes}.

The forcing used in this paper is a countable support iteration of the groupwise Silver forcings and the random forcing. Let $\mathbb{P}_{\kappa}=\{\mathbb{P}_{\xi}, \dot{\mathbb{Q}}_{\eta}:\xi\leq \kappa , \eta< \kappa\}$ be such a forcing of length $\kappa$. In lemma \ref{999} below, we will show that assuming \textbf{MA} in the ground model, any $\Sigma^{1}_{2}$ set in the generic extension by $\mathbb{P}_{\kappa}$ can be uniformized by a Baire-measurable map in the ground model. In order to prove this we first need the following lemma. It is proved in \cite{ZapId} but we include the proof here for the convenience of the reader.
\begin{lemma}\label{126}
 Suppose $\mathcal{I}$ is a $\sigma$-ideal on a Polish space $X$ such that $\mathbb{P}_{\mathcal{I}}$ is proper. Let $Y$ be a Polish space and $p\in \mathbb{P}_{\mathcal{I}}$ forces that $\dot{B}$ is a Borel subset of $Y$. Then there is a Borel $\mathcal{I}$-positive condition $q\leq p$ and a ground model Borel set $D\subseteq q\times Y$ such that $q\Vdash \dot{D}_{\dot{r}_{gen}}=\dot{B}$.
\end{lemma}
\begin{proof}
The proof is carried out by induction on the Borel rank of $\dot{B}$. Since the forcing $\mathbb{P}_{\mathcal{I}}$ preserves $\aleph_{1}$ by possibly strengthening the condition $p$ we may assume that the Borel rank of $\dot{B}$ is forced to be $\leq \alpha$ for a fixed countable ordinal $\alpha$.
Let $\mathcal{M}$ be a countable elementary submodel of a large enough structure.

 Assume $\dot{B}$ is forced to be a closed set. Fix a countable base $\mathcal{O}$ for the topology of the space $Y$. Since $\mathbb{P}_{\mathcal{I}}$ is proper we can find \cite{ZapDes} a Borel $I$-positive set $q\leq p$ (in fact $q$ is the set of all $\mathcal{M}$-generic reals in $p$) and a ground model Borel function $f: q\mapsto \mathcal{P}(\mathcal{O})$ such that $q\Vdash \check{f}(\dot{r}_{gen})=\{O\in \mathcal{O}: \dot{B}\cap O=\emptyset\}$. Define $D=\{(x,y)\in q\times Y : y\notin \bigcup f(x)\}$. It is easy to check that $D$ is the required Borel set. The proof for open sets is similar.

Now suppose $p$ forces that $\dot{B}=\bigcup_{n}\dot{B}_{n}$ where $\dot{B}_{n}$'s are sets of lower Borel rank. Let $q=\{x\in p : x \text{ is $\mathcal{M}$-generic}\}$. Using the inductive assumption for each $n\in \mathbb{N}$ find a maximal antichain $A(n)\subset \mathbb{P}_{\mathcal{I}}$ below $p$, such that for every condition $s\in A(n)$ there is a Borel set $D(s,n)\subset s\times Y$ such that $s\Vdash \dot{D}(s,n)_{\dot{r}_{gen}}=\dot{B}(n)$. For every $n\in \mathbb{N}$ let $D(n)=\bigcup\{D(s,n) : s\in \mathcal{M}\cap A(n)\}\cap q\times Y\subset q\times Y$. The condition $q$ forces that the generic real $\dot{r}_{gen}$ belongs to exactly one condition in the antichain $\mathcal{M}\cap A(n)$ for every $n$. Therefore $\dot{B}(n)=\bigcup\{\check{D}(s,n) : s\in \mathcal{M}\cap A(n)\}_{\dot{r}_{gen}}= \dot{D}(n)_{\dot{r}_{gen}}$. Now the set $D=\bigcup_{n} D(n)$ is clearly a Borel subset of $q\times Y$ and $q$ forces that $\dot{B}= \dot{D}_{\dot{r}_{gen}}$.

The countable intersection case is a similar argument.
\end{proof}
The following lemma can be ignored in proving theorem \ref{1} since it immediately follows from the large assumption. Nevertheless it implies that in order to get local triviality of isomorphisms or even *-homomorphisms of FDD-algebras, corollary \ref{local}, no large cardinal assumption is necessary.
\begin{lemma}\label{999}
Assume \textbf{MA} holds in the ground model and $\mathbb{P}_{\kappa}$ is a countable support iteration of length $\kappa$ of proper forcings of the form $\mathbb{P}_{\mathcal{I}}$ with compact conditions. If $\dot{C}$ is a $\mathbb{P}_{\kappa}$-name for a  $\Sigma_{2}^{1}$ subset of $\mathbb{R}\times \mathbb{R}$ in the extension such that for every $\dot{x}\in \mathbb{R}$ the vertical section $\dot{C}_{\dot{x}}$ is non-empty, then there are $q\in\mathbb{P}_{\kappa}$ and a Baire-measurable map $h:\mathbb{R}\mapsto \mathbb{R}$ such that for every  $\mathbb{P}_{\kappa}$-name $\dot{x}$ for a real
\begin{equation}
\nonumber q\Vdash (\dot{x},\check{h}(\dot{x}))\in \dot{C}.
\end{equation}
\end{lemma}
\begin{proof}
Let $\mathcal{I}_{\xi}$ be the $\sigma$-ideal associated with $\dot{\mathbb{Q}}_{\xi}$.
Since $\Sigma_{2}^{1}$ sets are projections of $\Pi_{1}^{1}$ sets and \textbf{MA} implies that all $\Sigma^1_{2}$ sets have the property of Baire, it is enough to uniformize $\Pi_{1}^{1}$ sets. Assume some $p\in \mathbb{P}_{\kappa}$ forces that $\dot{C}$ is a $\Pi_{1}^{1}$ subset of $\mathbb{R}\times \mathbb{R}$. There is a $\mathbb{P}_{\kappa}$-name $\dot{B}$ for a Borel subset of $\mathbb{R}^{3}$ such that $p\Vdash \mathbb{R}^{2} - pr_{\{1,2\}}(\dot{B})=\dot{C}$ where  $pr_{\{1,2\}}$ is the projection on the first and second coordinates of $\mathbb{R}^{3}$. Let the countable set $S\subset \kappa$ denote the  support of $p$ and

\begin{equation}
\nonumber \mathbb{P}_{S}=\{\mathbb{P}_{\xi}, \dot{\mathbb{Q}}_{\eta}:\xi\in S , \eta\in S\}
\end{equation}

and let $\mathcal{I}^{S}=\prod_{\xi\in S} \mathcal{I}_{\xi}$. Since these forcings are proper Suslin and $\omega^{\omega}$-bounding \cite[Lemma 4.3]{FaSh} we have  $p\Vdash_{\mathbb{P}_{S}} \mathbb{R}^{2} - pr_{\{1,2\}}(\dot{B})=\dot{C}$.

 Let $\alpha$ be the order-type of $S$. By forcing equivalence of $\mathbb{P}_{S}$ and $\mathbb{P}_{\mathcal{I}^{S}}=\mathcal{B}(\mathbb{R}^{S})/\mathcal{I}^{S}$ and for simplicity assume $p\in \mathbb{P}_{\mathcal{I}^{S}}$. Since $\mathbb{P}_{\mathcal{I}^{S}}$ is proper, by lemma \ref{126}, there is a ground model Borel set $D\subseteq \mathbb{R}^{\alpha}\times \mathbb{R}^{3}$
 and $q\leq p$ such that $q\Vdash \dot{B}=\dot{D}_{\dot{r}_{gen}}$ where $\dot{r}_{gen}$ is the canonical $\mathbb{P}_{\mathcal{I}^{S}}$-name for the generic real in $\mathbb{R}^{\alpha}$. Therefore

 \begin{equation}
  q\Vdash  \mathbb{R}^{2} - pr_{\{\alpha+1,\alpha+2\}}(\dot{D}_{\dot{r}_{gen}})=\dot{C}.
 \end{equation}

Now since the set $E=\mathbb{R}^{\alpha+2} -  pr_{\{1,\dots,\alpha+2\}}(D)$ is $\Pi_{1}^{1}$, by Kond\^{o}'s uniformization theorem, $E$ has a $\Pi_{1}^{1}$ and hence a  Baire-measurable uniformization $g: pr_{\{1,\dots,\alpha+1\}}(E)\mapsto \mathbb{R}$.


Let $\mathcal{M}$ be an elementary submodel of some large enough structure containing $\mathcal{I}^{S}$ and $\mathbb{P}_{\kappa}$, and also let $t=\{x\in q : \text{$x$ is $\mathcal{M}$-generic}\}$. Since $\mathbb{P}_{\mathcal{I}^{S}}$ is proper, $t$ is a condition in $\mathbb{P}_{\mathcal{I}^{S}}$.
Fix $x\in t$ and note that since the sections of $\dot{C}$ are non-empty, for every $y\in \mathbb{R}$  we have

\begin{equation}
\nonumber [pr_{\{\alpha+1,\alpha+2\}}(\dot{D}_{x})]_{y}=\dot{C}_{y}\neq \emptyset
\end{equation}

Therefore  $t\times \mathbb{R}\subseteq dom(g)$. For every $x\in t$ and $y\in \mathbb{R}$ we have  $(x, y, g(x,y))\in E$.

 Define the function $h: \mathbb{R}\mapsto \mathbb{R}$ by
\begin{equation}
\nonumber h(y)=g(\dot{r}_{gen},y)
\end{equation}

By above and (1) we have $t\Vdash (\dot{y},\check{h}(\dot{y}))\in \dot{C}$.
\end{proof}

\section{topologically trivial automorphisms of analytic p-ideal quotients of fdd-algebras}

In this section we study the automorphisms of quotients of FDD-algebras over ideals associated with analytic P-ideals with  Baire-measurable representations. Our result resembles the fact that for an analytic P-ideal $\mathcal{J}$ any automorphism of $P(\mathbb{N})/\mathcal{J}$ with a Baire-measurable representation has an asymptotically additive representation (see \cite{FaAn}, \S 1.5).
\begin{definition}
A map $\mu : \mathcal{P}(\mathbb{N})\rightarrow [0,\infty]$ is a submeasure supported by $\mathbb{N}$ if for $A, B\subseteq \mathbb{N}$

\begin{eqnarray}
 \nonumber &\mu(\emptyset)=0\\
 \nonumber &\mu(A)\leq \mu(A\cup B)\leq\mu(A)+\mu(B).
\end{eqnarray}
It is lower semicontinuous if for all $A\subseteq \mathbb{N}$ we have

\begin{equation}
\nonumber \mu(A)= \lim_{n\rightarrow\infty}\mu(A\cap [1,n]).
\end{equation}
\end{definition}
For a lower semicontinuous submeasure $\mu$ let

\begin{equation}
\nonumber Exh(\mu)=\{A\subseteq\mathbb{N}~:~\lim_{n}\mu(A\setminus [1,n])=0\}.
\end{equation}

This is an $F_{\sigma\delta}$ P-ideal on $\mathbb{N}$ (see \cite{FaAn}) and by Solecki's theorem \cite{Sol} every analytic P-ideal is of the form $Exh(\mu)$ for some  lower semicontinuous submeasure $\mu$.

For the rest of this section let $\mathcal{J}=Exh(\mu)$ be an analytic P-ideal on $\mathbb{N}$ for a lower semicontinuous submeasure $\mu$, containing all finite sets ($\mathcal{F}in\subseteq \mathcal{J}$). For each $a\in\mathcal{D}[\vec{E}]$ define $supp(a)\subseteq\mathbb{N}$ by
\begin{equation}\nonumber
 supp(a)=\{n\in\mathbb{N}: P_{n}a\neq 0\}
\end{equation}
 and in order to make notations simpler let $\hat{\mu}: \mathcal{D}[\vec{E}]\rightarrow [0, \infty]$ be $\hat{\mu}(a)= \mu(supp(a))$.

\begin{definition}[Approximate *-homomorphism] Assume $A$ and $B$ are unital C*-algebras. A map $\Psi: A \rightarrow B$ is an $\epsilon$-approximate unital *-homomorphism if for every
$a$ and $b$ in $A_{\leq 1}$ the following hold:

\begin{enumerate}
\item $\parallel \Psi(ab)-\Psi(a)\Psi(b)\parallel\leq \epsilon$
\item  $\parallel \Psi(a+b) - \Psi(a)-\Psi(b)\parallel \leq \epsilon$
\item $\parallel \Psi(a^{*})- \Psi(a)^{*}\parallel \leq \epsilon$
\item $|\| \Psi(a)\| - \| a \||\leq \epsilon$
\item $\parallel\Psi(I) - I\parallel\leq \epsilon$
\end{enumerate}

We say $\Psi$ is $\delta$-approximated by a unital *-homomorphism $\Lambda$ if $\parallel\Psi(a) - \Lambda(a)\parallel\leq \delta$ for all $a\in A_{\leq 1}$.
\end{definition}
 Next lemma is an Ulam-stability type result for finite-dimensional C*-algebras which will be required in the proof of lemma \ref{5}. To see a proof look at  \cite[Theorem 5.1]{FaCalkin}.
\begin{lemma}\label{21}
There is a universal constant $K < \infty$ such that for every $\epsilon$ small enough , $A$ and $B$ finite-dimensional C*-algebras, every  Borel-measurable $\epsilon$-approximate unital *-homomorphism $\Psi: A \rightarrow B$ can be $K\epsilon$-approximated by a unital *-homomorphism.
\end{lemma}

We will also use the following standard fact. To see a proof of this refer to \cite[Theorem 5.8]{FaCalkin}.
\begin{lemma}\label{22}
If $0 < \epsilon< 1/8$ then in every C*-algebra $A$ the following holds.
For every $a\in A$ satisfying $\| a- a^{2}\|\leq\epsilon$ and $\| a- a^{*}\|\leq\epsilon$, there is a projection $P\in A$
such that $\| P - a\| \leq 4\epsilon$.
\end{lemma}

  Assume $\Phi : \mathcal{C}^{\mathcal{J}}[\vec{E}]\rightarrow \mathcal{C}^{\mathcal{J}}[\vec{E}]$ is an automorphism and  $\mathcal{D}[\vec{E}]$ is equipped with the strong operator topology. Recall that if $M\subseteq \mathbb{N}$ then  $P_{M}$ denotes the projection on the closed span of $\bigcup_{n\in M}  \{e_{i} : i \in \vec{E}_{n}\}$. For each $n$ fix  a finite set  of operators $G_{n}$
  which is $2^{-n}$-dense (in norm) in the unit ball of $\mathcal{D}_{\{n\}}[\vec{E}]\cong \mathbb{M}_{|E_{n}|}(\mathbb{C})$. Let $F=\prod_{n=1}^{\infty} G_{n}$ and $F_{M}= P_{M}F$ for any $M\subseteq \mathbb{N}$.

\begin{lemma}\label{5}
  If an automorphism $\Phi : \mathcal{C}^{\mathcal{J}}[\vec{E}]\rightarrow \mathcal{C}^{\mathcal{J}}[\vec{E}]$ has a Baire-measurable representation $\Phi_{*}$, then it has a *-homomorphism representation.
\end{lemma}
\begin{proof}
First we show that $\Phi$ has a (strongly) continuous representation on $F$ and then we construct a *-homomorphism representation on $\mathcal{D}[\vec{E}]$ by using a similar argument used in chapter 6 of \cite{FaCalkin}.

  The first part is a well-known fact  (see \cite{FaAn}). To see this let $G$ be a dense $G_{\delta}$ set such that the restriction of $\Phi_{*}$ is continuous on $G$ and $G= \bigcap_{i=1}^{\infty}U_{i}$ where $U_{i}$ are dense open sets in $F$. Assume $U_{i+1}\subseteq U_{i}$ for each $i$. Recursively choose $1=n_{1}\leq n_{2}\leq \dots$ and $s_{i}\in  F_{[n_{i},n_{i+1})}$ such that for every $a\in F$ if ${P}_{[n_{i},n_{i+1})}a= s_{i}$ then $a\in U_{i}$. Now let

\begin{equation}
\nonumber  t_{0}= \sum_{i} s_{2i} \quad\quad  t_{1}= \sum_{i} s_{2i+1}\\
\end{equation}
Let $Q_{0}=\sum_{i~even}{P}_{[n_{i},n_{i+1})}$ and $Q_{1}=\sum_{i~odd}{P}_{[n_{i},n_{i+1})}$.  Define $\Psi$ on $F$ by
\begin{align}
\nonumber\Psi(a)=\Psi_{0}(a) + \Psi_{1}(a)
\end{align}
where
\begin{align}
\nonumber \Psi_{0}(a)= \Phi_{*}(Q_{0}a + t_{1}) - \Phi_{*}(t_{0})\\
\nonumber \Psi_{1}(a)= \Phi_{*}(Q_{1}a + t_{0}) - \Phi_{*}(t_{1}).
\end{align}

It is easy to see that $\Psi$ is a continuous representation of $\Phi$ on $F$.
By possibly replacing $\Psi$ with the map $a\rightarrow\Psi(a)\Psi(I)^{*}$ we can assume $\Psi$ is unital.

In order to find a *-homomorphism representation of $\Phi$, first we find a representation of $\Phi$ which is \emph{stabilized} by a sequence $\{u_{n}\}$ of orthogonal elements of $F$ in the sense to be made clear below.

\textbf{Claim 1.}
  For all $n$  and $\epsilon>0$ there are $k>n$ and $u\in F_{[n,k)}$ such that for every $a$ and $b$ in $F$ satisfying ${P}_{[n,\infty)}a={P}_{[n,\infty)}b$ and ${P}_{[n,k)}a={P}_{[n,k)}b =u$, there exists $c\in \mathcal{D}[\vec{E}]$ such that $\|\Psi(a)-\Psi(b)- c\|<\epsilon$ and $\hat{\mu}({P}_{[k,\infty)}c)<\epsilon$.
\begin{proof}
Suppose claim fails for $n$ and $\epsilon>0$. Recursively build sequences $m_{i}, u_{i}, s_{i}$ and $t_{i}$ for $i\in \mathbb{N}$ as follows
 \begin{enumerate}
 \item[(a)] $n=m_{0}< m_{1}<m_{2}< \dots$,
 \item[(b)] $u_{i}\in F_{[m_{i},m_{i+1})}$,
 \item[(c)] $s_{i}$ and $t_{i}$ are elements of $F_{[0,n)}$,
 \item[(d)] for every $c\in\mathcal{D}[\vec{E}]$ if $\|\Psi(s_{i}+u_{i})-\Psi(t_{i}+u_{i})- c\|<\epsilon$ then  $\hat{\mu}({P}_{[i,\infty)}c)\geq\epsilon$.
 \end{enumerate}
 This can be easily done by our assumption.
  Since $F_{[0,n)}$ is finite let $\langle s,t\rangle$ be a pair $\langle s_{i},t_{i}\rangle$ which appears infinitely often.
  Note that $\Psi$ is a representation of an automorphism, therefore we can find $k$ large enough and $d, h\in \mathcal{D}^{\mathcal{J}}[\vec{E}]$ such that for every $j\in\mathbb{N}$

\begin{eqnarray}
\nonumber   \|\Psi(s+\sum_{i}u_{i})-\Psi(s+u_{j})-\Psi(\sum_{i\neq j}u_{i})-d\|&<&\epsilon/3\\
 \nonumber  \|\Psi(t+\sum_{i}u_{i})-\Psi(t+u_{j})-\Psi(\sum_{i\neq j}u_{i})-h\|&<&\epsilon/3,
\end{eqnarray}
and
\begin{equation}
\hat{\mu}({P}_{[k,\infty)}d)\leq\epsilon/3,~~~~ \qquad \qquad ~~~~ \hat{\mu}({P}_{[i,\infty)}h)\leq\epsilon/3.
\end{equation}
Both $d$ and $h$ can be chosen to be $\Psi(0)$. Also fix a $c\in\mathcal{D}[\vec{E}]$ such that

\begin{equation}
 \|\Psi(s+\sum_{i}u_{i})-\Psi(t+\sum_{i}u_{i}) - c\|<\epsilon/3.
 \end{equation}
 We will see that with these assumptions no such $c$ could belong to $\mathcal{D}^{\mathcal{J}}[\vec{E}]$.
  For infinitely many $j\geq k$ we have

\begin{eqnarray}
\nonumber \|\Psi(s+u_{j})&-&\Psi(t+u_{j})-(d+h+c)\| \\
\nonumber&\leq& \|\Psi(s+\sum_{i}u_{i})-\Psi(s+u_{j})-\Psi(\sum_{i\neq j}u_{i})-d\|\\
\nonumber &+& \|\Psi(t+\sum_{i}u_{i})-\Psi(t+u_{j})-\Psi(\sum_{i\neq j}u_{i})-h\|\\
\nonumber &+& \|\Psi(s+\sum_{i}u_{i})-\Psi(t+\sum_{i}u_{i}) - c\|\\
\nonumber &<& \epsilon/3+\epsilon/3+\epsilon/3=\epsilon.
\end{eqnarray}

Hence by condition (d) we have $\hat{\mu}({P}_{[j,\infty)}(d+h+c))\geq\epsilon$ and

\begin{equation}
\nonumber \hat{\mu}({P}_{[j,\infty)}d)+\hat{\mu}({P}_{[j,\infty)}h)+\hat{\mu}({P}_{[j,\infty)}c)\geq\hat{\mu}({P}_{[j,\infty)}(d+h+c))\geq\epsilon.
\end{equation}

Therefore by (2) we have $\hat{\mu}({P}_{[j,\infty)}c)\geq\epsilon$ for infinitely many $j\geq k$.
 Since $c$ was arbitrary this implies that for any $c$ satisfying (3) we have $\lim_{i\rightarrow\infty}\hat{\mu}({P}_{[i,\infty)}c)>\epsilon$. Hence
$\Psi(s+\sum_{i}u_{i})- \Psi(t+\sum_{i}u_{i})$ does not belong to $\mathcal{D}^{\mathcal{J}}[\vec{E}]$.
 This is a contradiction since $(s+\sum_{i}u_{i}) - (t+\sum_{i}u_{i})$ is a compact operator and therefore $\Psi(s+\sum_{i}u_{i}) - \Psi(t+\sum_{i}u_{i})\in\mathcal{D}^{\mathcal{J}}[\vec{E}]$.
\end{proof}

We build two increasing sequences of natural numbers $(n_{i})$ and $(k_{i})$ such that $n_{i}< k_{i}< n_{i+1}$ for every $i$ and so called "stabilizers" $u_{i}\in F_{[n_{i},n_{i+1})}$ such that for all $a, b \in F$ which $ {P}_{[n_{i}, n_{i+1})}a=  {P}_{[n_{i},n_{i+1})}b = u_{i}$ the following holds:
\begin{enumerate}
\item If ${P}_{[n_{i+1},\infty)}a={P}_{[n_{i+1},\infty)}b$ then there exists $c\in \mathcal{D}[\vec{E}]$ such that $\|[\Psi(a)-\Psi(b)] {P}_{[k_{i},\infty)} -c\|<2^{-n_{i}}$ and $\hat{\mu}({P}_{[k_{i},\infty)}c)<2^{-n_{i}}$.\\
\item If $ {P}_{[0,n_{i})}a=  {P}_{[0,n_{i})}b$ then $\| [\Psi(a) - \Psi(b)]{P}_{[k_{i},\infty)}\| \leq 2^{-n_{i}}$.
\end{enumerate}

Assume $ n_{i}, k_{i-1}$ and $u_{i-1}$ have been chosen. By the claim above we can find $k_{i}$ and $u_{i}^{0}\in F_{[n_{i},k_{i})}$ such that $(1)$ holds. Now since $\Psi$ is strongly continuous we can find $n_{i+1}\geq k_{i}$ and $u_{i}\in F_{[n_{i},n_{i+1})}$ extending $u_{i}^{0}$ such that $(2)$ holds.

Let $J_{i}=[n_{i}, n_{i+1})$ and $\nu_{i} = \mathcal{D}_{J_{i}}[\vec{E}]$. Then $\mathcal{D}[\vec{E}]= \prod \nu_{i}$ and for $b \in \mathcal{D}[\vec{E}]$ we have $b=\sum_{j} b_{j} $ where $b_{j}\in \nu_{j}$. Note that $F_{J_{i}}$ is finite and $2^{-n_{i}+1}$-dense in $\nu_{i}$. Fix a linear ordering of $F_{J_{i}}$ and define $\sigma_{i}:~ \nu_{i} \longrightarrow  F_{J_{i}}$ by letting $\sigma_{i}(b)$ to be the least element of $F_{J_{i}}$ which is in the $2^{-n_{i}+1}$- neighborhood of $b$. For $b\in \mathcal{D}[\vec{E}]_{\leq 1}$ let $b_{even}= \sum \sigma_{2i}(b_{2i})$ and $b_{odd}= \sum \sigma_{2i+1}(b_{2i+1})$. Both of these elements belong to $F$ and $b- b_{even}- b_{odd}$ is compact.\\
 Define $\Lambda _{2i+1}: \nu_{2i+1}\longrightarrow \mathcal{D}[\vec{E}]$ by

\begin{equation}
\nonumber\Lambda_{2i+1}(a) = \Psi(u_{even} + \sigma_{2i+1}(a))- \Psi(u_{even}).
\end{equation}

Since $\Psi$ is continuous and $\sigma_{i}$ is Borel-measurable, $\Lambda_{2i+1}$ is Borel-measurable. Let $Q_{i}={P}_{[k_{i-1},k_{i+1})}$ with $k_{-1}=0$. Note that  if $\mid i-j\mid>1$ then $Q_{i}$ and $Q_{j}$ are orthogonal.\\
Let  $\Lambda~:~ \prod_{i=0}^{\infty}\nu_{2i+1}\longrightarrow \mathcal{D}[\vec{E}]$ be defined by
\begin{equation}
\nonumber\Lambda(b) = \Psi(u_{even} + b_{odd})- \Psi(u_{even}).
\end{equation}
Since $b-b_{odd}$ ia compact we have $\Psi(b)-\Lambda(b)\in \mathcal{D}^{\mathcal{J}}[\vec{E}]$. Therefore $\Lambda$ is a representation of $\Phi$ on $\prod_{i=0}^{\infty}\nu_{2i+1}$.

\textbf{Claim 2.} For $b=\sum_{j} b_{2j+1}\in \prod_{j=0}^{\infty} \nu_{2j+1}$, the operator $\Psi(b)-\sum_{i=0}^{\infty} Q_{2i+1}\Lambda_{2i+1}(b_{2i+1})$
 belongs to $\mathcal{D}^{\mathcal{J}}[\vec{E}]$.

Since $\Lambda$ is a representation of $\Phi$ on $\prod_{i=0}^{\infty}\nu_{2i+1}$, there exists $c\in \mathcal{D}^{\mathcal{J}}[\vec{E}]$ such that for every large enough $l$,  $\|[\Psi(b)+\Lambda(b)]Q_{2l+1}-c\|<2^{-n_{2l}}$ and $\hat{\mu}({P}_{[k_{2l},\infty)}c)<2^{-n_{2l}}$.
   Let $b^{l}=\sum_{j=l}^{\infty}\sigma_{2j+1}(b_{2j+1})$ and apply (1) to $b^{l}$ and $b_{odd}$ implies that there exists $c^{\prime}\in \mathcal{D}[\vec{E}]$ such that  $\|[\Psi(u_{even}+ b^{l})- \Psi(u_{even}-b_{odd})]Q_{2l+1}-c^{\prime}\|<2^{-n_{2l}}$ and $\hat{\mu}({P}_{[k_{2l},\infty)}c^{\prime})<2^{-n_{2l}}$.
   Therefore

 \begin{align} \nonumber
 &\| Q_{2l+1}[\Psi(b)-\sum_{i}Q_{2i+1}\Lambda_{2i+1}(b_{2i+1})]-(c+c^{\prime})\|\\  \nonumber
 &\leq\|Q_{2l+1} [\Psi(b)- \Lambda(b)]-c\| + \| Q_{2l+1}[\Lambda(b)- \sum_{i}Q_{2i+1}\Lambda_{2i+1}(b_{2i+1})]-c^{\prime}\|\\ \nonumber
 &\leq 2^{-n_{2l}}+ \| Q_{2l+1}[\Lambda(b)-\Psi(u_{even}+ b^{l})-\Psi(u_{even})]-c^{\prime}\| \\ \nonumber
 & + \| Q_{2l+1}[\Psi(u_{even}+ b^{l})-\Psi(u_{even})-\Lambda_{2l+1}(b_{2l+1})]\| \qquad \text{[Apply (2)]}\\\nonumber
 &\leq 3. 2^{-n_{2l}}\nonumber.
 \end{align}
  Now for $d={P}_{[k_{2l},\infty)}(c+c^{\prime})+[\sum_{n=0}^{l}Q_{2n+1}\Lambda_{2n+1}(b_{2n+1})-{P}_{[0,k_{2l})}\Psi(b)]$ and any large enough $l$ we have
\begin{equation}
\nonumber \|\Psi(b)-\sum_{i}Q_{2i+1}\Lambda_{2i+1}(b_{2i+1})-d\|\leq\sum_{j=l}^{\infty}2^{-n_{2j}}
\end{equation}
and $\hat{\mu}({P}_{[k_{2l},\infty)}d)<2.2^{-n_{2l}}$. This completes the proof of the claim 2.

Now let $\Lambda^{\prime}_{2i+1}: \nu_{2i+1}\rightarrow Q_{2i+1}\mathcal{D}[\vec{E}]$ be defined as
\begin{equation}
\nonumber \Lambda^{\prime}_{2i+1}(b)=Q_{2i+1}\Lambda_{2i+1}(b).
\end{equation}

 Let $c_{2i+1}=\Lambda_{2i+1}^{\prime}(I_{2i+1})$, where $I_{2i+1}$ is the unit of $\nu_{2i+1}$, and $\delta_{i}= max \{\| c_{2i+1}^{2} - c_{2i+1}\| , \| c_{2i+1}^{*} - c_{2i+1}\|\}$. We show that $\limsup_{i}\delta_{i}= 0$. Assume not; find $\delta> 0$ and an infinite set $M\subset 2\mathbb{N}+1$ such that for all $i\in M$ we have $max\{\| c_{i}^{2}-c_{i}\| , \| c_{i}^{*} - c_{i}\|\}>\delta$. Let $c = \sum_{i\in M} c_{i}$, by our previous claim if $P=\sum_{i\in M} Q_{i}$ then $\Psi(P) - c$ is compact. Therefore $c - c^{2}$ and $c - c^{*}$ are compact. Since $c_{i}$'s are orthogonal we have $c^{2}= \sum_{i\in M} c_{i}^{2}$ and $c^{*}= \sum_{i\in M} c_{i}^{*}$. Thus for large enough $i\in M$ we have $\| c_{i} - c_{i}^{2} \|= \| Q_{i}(c - c^{2})\|\leq \delta$ and $\| c_{i} - c_{i}^{*} \|= \| Q_{i}(c - c^{*})\|\leq \delta$, which is a contradiction.\\
Applying lemma \ref{22} to $c_{2i+1}$ for large enough $i$  we get projections $S_{2i+1}\leq Q_{2i+1}$ such that $\limsup_{i\rightarrow\infty} \| S_{2i+1} -  \Lambda_{2i+1}^{\prime}(I_{2i+1})\| = 0$. Let
\begin{equation}
\nonumber\Lambda_{i}^{\prime\prime}(a) = S_{2i+1}\Lambda_{2i+1}^{\prime}(a)S_{2i+1}
\end{equation}
for $a\in \nu_{2i+1}$.
Now by re-enumerating indices we can assume $\Lambda_{i}^{\prime\prime}$ is $\epsilon$-approximate unital *-homomorphism, for small enough $\epsilon$.
Then $\Lambda^{\prime\prime} (a)= \sum_{i} \Lambda_{i}^{\prime\prime}(a)$ is a representation of $\Phi$ on $\prod_{i}\nu_{2i+1}$.
Let
\begin{align}
&\nonumber \delta_{i}^{0} = \sup_{a,b\in \nu_{2i+1}\leq 1}\{ \| \Lambda_{i}^{\prime\prime}(ab) - \Lambda_{i}^{\prime\prime}(a)\Lambda_{i}^{\prime\prime}(b)\|\} \\
&\nonumber \delta_{i}^{1} = \sup_{a,b\in \nu_{2i+1}\leq 1} \{\| \Lambda_{i}^{\prime\prime}(a+b) - \Lambda_{i}^{\prime\prime}(a)-\Lambda_{i}^{\prime\prime}(b)\|\} \\
&\nonumber \delta_{i}^{2} = \sup_{a \in \nu_{2i+1}\leq 1}\{\| \Lambda_{i}^{\prime\prime}(a^{*}) - \Lambda_{i}^{\prime\prime}(a)^{*}\| \} \\
&\nonumber \delta_{i}^{3} = \sup_{a \in \nu_{2i+1}\leq1} \{\| \Lambda_{i}^{\prime\prime}(a)\| - \| a\| \}. \\
\end{align}

We claim that $\lim_{i} max_{0\leq k\leq 3}\delta_{i}^{k} =0.$ We only show $\lim_{i} \delta_{i}^{0}  =  0$ since the others are similar.  Take $a$ and $b$ in $\sum_{i} \nu_{2i+1}$ such that ${P}_{J_{i}}a = a_{i}$ and ${P}_{J_{i}}b = b_{i}$ for all $i$. Since $\Psi(ab) - \Psi(a)\Psi(b)$ is compact, by claim 2 so is $\Lambda^{\prime\prime}(ab) -\Lambda^{\prime\prime}(a)\Lambda^{\prime\prime}(b)$ , which implies $\lim \delta_{i}^{0}=0$.
Let $\delta_{j}= max_{0\leq i\leq3}\{\delta_{j}^{i}\}$. Each $\Lambda_{j}^{\prime\prime}$ is a  Borel measurable $\delta_{j}$-approximate *-homomorphism. Therefore by lemma \ref{21} for any  large enough $j$ we can find a *-homomorphism $\Theta_{j}$ defined on $\nu_{2j+1}$ which is $K\delta_{j}$-approximation  of $\Lambda_{j}^{\prime\prime}$. Define $\Theta ~:~ \sum_{i} \nu_{2i+1}\longrightarrow \mathcal{D}[\vec{E}]$ by $\Theta = \sum \Theta_{i}$. Since $\lim_{j} \delta_{j} = 0$, $\Theta$ is a representation of $\Phi$ on $ \sum_{i>n} \nu_{2i+1}$. Hence $\Theta$ can be extended to a *-homomorphism representation of $\Phi$ on $ \sum_{i} \nu_{2i+1}$.
By repeating the same argument for even intervals instead of odd intervals, one can get a *-homomorphism representation of  $\Phi$ on $ \sum_{i} \nu_{2i}$. Now by combining these two representation we get the desired representation of $\Phi$.\\
\end{proof}



\section{automorphisms of borel quotients of fdd-algebras are topologically trivial}

This section is devoted to find local Baire-measurable representations of $\Phi$. For this section it is enough to assume $\mathcal{J}$ is a Borel ideal on natural numbers containing all finite sets and we also assume that all elements of the FDD-algebra are taken from the unit ball.
 We say an automorphism $\Phi:\mathcal{C}^{\mathcal{J}}[\vec{E}]\rightarrow \mathcal{C}^{\mathcal{J}}[\vec{E}]$ is trivial if it has a representation which is *-homomorphism and that it is $\Delta^{1}_{2}$ if the set $\{(a,b) : \Phi(\pi_{\mathcal{J}}(a)) =  \pi_{\mathcal{J}}(b)\}$ is $\Delta^{1}_{2}$.

Fix  a partition $\vec{I} = (I_n)$ of natural numbers into finite intervals  and for each $n$ fix  a finite set $G_{n}$ of operators
  which is $2^{-n}$-dense (in norm) in the unit ball of $\mathcal{D}_{\{n\}}[\vec{E}]\cong \mathbb{M}_{|E_{n}|}(\mathbb{C})$. As before let
 \begin{equation}
 \nonumber F_{n}^{\vec{I}}=\prod_{i\in I_{n}}G_{i}, \qquad\qquad~~~~~~~~ F^{\vec{I}}=\prod_{n\in\mathbb{N}}F_{n}^{\vec{I}}
 \end{equation}
   and for $M\subset \mathbb{N}$ let
   \begin{equation}
   \nonumber F_{M}^{\vec{I}}=\prod_{n\in M}F_{n}^{\vec{I}}.
   \end{equation}

   Note that  each $F_{n}^{\vec{I}}$ is $2^{k-1}$-dense in $\mathcal{D}_{I_{n}}[\vec{E}]$ where $k$ is the smallest element of $I_{n}$.
 Since each $G_{n}$ is finite, the product topology and the strong operator topology coincide on $F^{\vec{I}}$. For any $M\subseteq\mathbb{N}$  let $\hat{P}_{M}= P_{\cup_{n\in M}I_{n}}$.
 \begin{lemma}
 If a forcing notion $\mathbb{P}$ captures $F^{\vec{I}}$, then there is a $\mathbb{P}$-name $\dot{x}$ for a real such that for every $p\in\mathbb{P}$ there is an infinite $M\subset \mathbb{N}$ such that for every $a\in \mathcal{D}_{\cup_{n\in M}I_{n}}[\vec{E}]$ there is $q_{a}\leq p$ such that $q_{a}\Vdash \hat{P}_{M}\dot{x}  =^{\mathcal{J}} \check{a}$.
 \end{lemma}
 \begin{proof}
 Since  the ideal $\mathcal{J}$ contains all finite sets and the sequence $\{F_{n}^{\vec{I}}\}$ is eventually dense in $\mathcal{D}_{\cup_{n\in M}I_{n}}[\vec{E}]$. The proof follows from the definition \ref{125}.
 \end{proof}

Let $\mathcal{C}_{M}[\vec{E}]=\mathcal{D}_{M}[\vec{E}]/\mathcal{D}_{M}[\vec{E}]\cap\mathcal{D}^{\mathcal{J}}{\vec{E}}]$ and define the following ideals on $\mathbb{N}$.

\begin{eqnarray}
\nonumber &Triv_{\Phi}^{0}&=\{ M\subset\mathbb{N}:~  \Phi\upharpoonright \mathcal{C}_{M}[\vec{E}] ~ \text{has a strongly continuous representation}\}\\
\nonumber &Triv_{\Phi}^{1}&= \{M \subset\mathbb{N}:~  \Phi\upharpoonright \mathcal{C}_{M}[\vec{E}] ~is~ \Delta_{2}^{1}\}.
\end{eqnarray}

We say that $\Phi$ is \emph{locally topologically trivial} if $Triv_{\Phi}^{0}$ is non-meager and it is \emph{locally} $\Delta_{2}^{1}$ if $Triv_{\Phi}^{1}$ is non-meager.\\

 The following lemma is well-known and is proved in \cite[lemma 4.5]{FaSh}, where $\mathbb{P}$ is countable support iteration of some creature forcings and the random forcing. Since groupwise Silver forcings as well as random forcing are also Suslin proper, $\omega^{\omega}$-bounding and have continuous reading  of names the same proof works for $\mathbb{P}=\{\mathbb{P}_{\xi}, \dot{\mathbb{Q}}_{\eta}~ :~ \xi\leq \kappa , \eta< \kappa\}$, a countable support iteration of forcings such that each  $\dot{\mathbb{Q}}_{\eta}$ is forced to be either some groupwise Silver forcing  or the random forcing.
\begin{lemma}\label{222}
Assume $\mathbb{P}=\{\mathbb{P}_{\xi}, \dot{\mathbb{Q}}_{\eta}~ :~ \xi\leq \kappa , \eta< \kappa\}$ is as above and $\dot{x}$ is a $\mathbb{P}$-name for a real. For $A\subseteq \mathbb{R}$ a Borel set and $g:\mathbb{R}^{2}\rightarrow \mathbb{R}$ a Borel function, if $p\in\mathbb{P}$ is such that $\dot{x}$ is continuously read below $p$, then the set
\begin{equation}
\nonumber \{a: ~ p\Vdash g(\check{a},\dot{x})\in A \}
\end{equation}
is $\Delta_{2}^{1}$.
\end{lemma}

Note that since $\mathbb{P}$ is $\omega^{\omega}$- bounding we can assume all partitions of $\mathbb{N}$ into finite intervals in the generic extension by $\mathbb{P}$ are ground model partitions. We will use lemma \ref{222} to show that if all partitions are captured by some groupwise Silver forcings in stationary many steps of uncountable cofinality, then any automorphism $\Phi$ is forced to be locally $\Delta_{2}^{1}$ in the generic extension.
\begin{lemma}\label{4}
Assume $\mathbb{P}$ is a countable support iteration forcing notion of length $\mathfrak{c}^{+}$ as above such that for every partition ${\vec{I}}$ of $\mathbb{N}$ into finite intervals the set
\begin{equation}
\nonumber \{\xi< \mathfrak{c}^{+} : ~\Vdash_{\mathbb{P}_{\xi}} \dot{\mathbb{Q}_{\xi}} ~\text{captures}~ F^{\vec{I}} ~\text{and} ~cf(\xi)\geq \aleph_{1}\}
\end{equation}
is stationary. Then every automorphism $\Phi : \mathcal{C}^{\mathcal{J}}[\vec{E}]\rightarrow \mathcal{C}^{\mathcal{J}}[\vec{E}]$ is forced to be locally $\Delta_{2}^{1}$.\\
\end{lemma}
\begin{proof}
Let $\dot{\Phi}$ be a $\mathbb{P}$-name for an automorphism in the generic extension as above and $\dot{\Phi}_{*}$ be an arbitrary representation of $\dot{\Phi}$. Let $G\subset \mathbb{P}$ be a generic filter.\\
Assume $Triv^{1}_{int_{G}(\dot{\Phi})}$ is meager in $V[G]$ with a witnessing partition $\vec{I}=(I_{n})$, i.e. for every infinite $A\subset\mathbb{N}$ the set $\bigcup_{n\in A} I_{n}$ is not in $Triv^{1}_{int_{G}(\dot{\Phi})}$. Since our forcings have cardinality $<\mathfrak{c}^{+}$,
the set of all $\xi< \mathfrak{c}^{+}$ of uncountable cofinality such that $\vec{I}$ witnesses $Triv^{1}_{int_{G\upharpoonright \xi}(\dot{\Phi}\upharpoonright\xi)}$ is meager in $V[G\upharpoonright\xi]$ includes a club $C$ (cf. \cite{FaSh}) relative to the set $\{\xi<\mathfrak{c}^{+}: cf(\xi)>\aleph_{0}\}$.\\
By our assumption there is a stationary set $S$ of ordinals of uncountable cofinalities such that for all $\xi \in S$ we have $\Vdash_{\mathbb{P}_{\xi}} " \dot{\mathbb{Q}}_{\xi} $ adds a real $\dot{x}_{\xi}$ which captures $F^{\vec{I}}"$.
 Fix $\eta\in S\cap C$. Let $\dot{y}$ be a $\mathbb{P}_{[\eta,\mathfrak{c}^{+}]}$-name such that

 \begin{equation}
 \nonumber \Phi(\pi_{\mathcal{J}}(\dot{x}_{\eta}))= \pi_{\mathcal{J}}(\dot{y}).
  \end{equation}
   Note that $\dot{x}_{\eta}$ is the generic real added by $\mathbb{Q}_{\eta}$ and since $\mathbb{P}$ has the continuous reading of names for any $p\in \mathbb{P}_{[\eta,\mathfrak{c}^{+}]}$ there are $q\leq p$, a countable set $S$ containing $\eta$, a compact set $K\subseteq \mathbb{R}^{S}$ and a continuous map $h : K\mapsto \mathbb{R}$ such that $q$ forces that $\check{h}(\langle \dot{x}_{\xi}: \xi\in S \rangle)= \dot{y}$. Since $\dot{\mathbb{Q}}_{\eta}$ captures $F^{\vec{I}}$ there is an infinite $A\subset \mathbb{N}$ such that if $M=\bigcup_{n\in A}I_{n}$ for every $a\in \mathcal{D}_{M}(\vec{E})$ there is $q_{a}\leq q$ such that $q_{a}\Vdash \check{P}_{M}\dot{x}_{\eta} =^{\mathcal{J}} \check{a}$ and therefore $ \Phi_{*}(\check{P}_{M})\dot{y}=^{\mathcal{J}} \Phi_{*}(\check{a})$. For every $a\in \mathcal{D}_{M}(\vec{E})$ we have

   \begin{equation}
   \nonumber \Phi_{*}(a)=^{\mathcal{J}}b \Longleftrightarrow q_{a}\Vdash b=^{\mathcal{J}} \Phi_{*}(\check{P}_{M})\check{h}(\langle \dot{x}_{\xi}: \xi\in S \rangle),
   \end{equation}
 so lemma \ref{222} implies that the set $\{(a,b)\in \mathcal{D}_{M}[\vec{E}]\times \mathcal{D}[\vec{E}] : \Phi(\pi_{\mathcal{J}}(a)) =  \pi_{\mathcal{J}}(b)\}$
  is $\Delta^{1}_{2}$.
 Therefore $M$ is in $Triv^{1}_{int_{G\upharpoonright \eta}(\dot{\Phi}\upharpoonright \eta)}$, which contradicts the assumption that $\vec{I}$ witnesses the meagerness of $Triv^{1}_{int_{G\upharpoonright \eta}(\dot{\Phi}\upharpoonright \eta)}$.
\end{proof}

The following lemmas is very similar to \cite[lemma 4.9]{FaSh}.
\begin{lemma}\label{11}
Suppose $f$ and $g$ are functions such that each of them is a representation of a *-homomorphism from $\mathcal{C}^{\mathcal{J}}[\vec{E}]$ into $\mathcal{C}^{\mathcal{J}}[\vec{E}]$. Assume
\begin{equation}
\nonumber \Delta_{f,g,\mathcal{J}} = \{ a\in F^{\vec{I}}~:~ f(a)\neq^{\mathcal{J}} g(a)\}
\end{equation}
is null. Then $\Delta_{f,g,\mathcal{J}}$ is empty.
\end{lemma}
\begin{proof}
By inner regularity of the Haar measure we can find a compact set $K\subset F$ disjoint from $\Delta_{f,g,\mathcal{J}}$ of measure $>$ 1/2. Fix any $a\in F^{\vec{I}}$. Since the set $K+a$ also has measure $>1/2$, we can find $b\in K$ such that $b+a$ is also in $K$. Now we have
\begin{equation}
\nonumber f(a)=^{\mathcal{J}}f(a+b)-f(b)=^{\mathcal{J}} g(a+b)-g(b)=^{\mathcal{J}}g(a)
\end{equation}
\end{proof}

\begin{corollary}\label{2}
Suppose $f$ and $g$ are continuous functions such that each of them is a representations of a *-homomorphism from $\mathcal{D}[\vec{E}]$ into $\mathcal{C}(A)$ and the random forcing $\mathcal{R}$ forces that $f(\dot{x})=^{\mathcal{J}} g(\dot{x})$, where $\dot{x}$ is the canonical name for the random real. Then $f(a)=^{\mathcal{J}}g(a)$ for every $a \in \mathcal{D}[\vec{E}]$.
\end{corollary}
\begin{proof}
Let $\Delta_{f,g,\mathcal{J}}$ be as defined in previous lemma. If $\Delta_{f,g,\mathcal{J}}$ is null by lemma \ref{11} we are done. Assume $\Delta_{f,g,\mathcal{J}}$ has positive measure and $M$ is a countable model of \textbf{ZFC} containing codes for $f,g$ and $\mathcal{J}$, since $\dot{x}$ is the random real, $\dot{x}\in \Delta_{f,g,\mathcal{J}}$ and therefore $f(\dot{x})\neq^{\mathcal{J}} g(\dot{x})$ in the generic extension. But our assumption $f(\dot{x})=^{\mathcal{J}} g(\dot{x})$ is a $\Delta^{1}_{1}$ statement so it is true in $V$. Which is a contradiction.
\end{proof}
Recall that for $M\subset \mathbb{N}$, $P_{M}$ is the projection on the closed span of $\bigcup_{n\in M}  \{e_{i} : i \in \vec{E}_{n}\}$.
\begin{lemma}\label{3}
Suppose $\mathcal{J}$ is a Borel ideal on $\mathbb{N}$. If $a\in \mathcal{D}[\vec{E}]\setminus \mathcal{D}^{\mathcal{J}}[\vec{E}]$  and $\mathcal{L}$ is a non-meager ideal on $\mathbb{N}$, then there exists $M\in\mathcal{L}$ such that $P_{M}a\notin \mathcal{D}^{\mathcal{J}}[\vec{E}]$.
\end{lemma}
\begin{proof}
 Since $a$ does not belong to $\mathcal{D}^{\mathcal{J}}[\vec{E}]$ there is $\epsilon>0$ such that
 \begin{equation}
 \nonumber A=\{n\in \mathbb{N} : \|a_{n}\|>\epsilon\}\notin \mathcal{J}.
 \end{equation}
 Since $A\cap \mathcal{J}$ is a proper Borel ideal on $A$ there are disjoint finite sets  $I_{n}$  such that $\bigcup_{n\in\mathbb{N}} I_{n}=A$ and for every infinite $X\subseteq \mathbb{N}$ the set $\bigcup_{n\in X} I_{n}\notin A\cap\mathcal{J}$. Let $\vec{J}=(J_{n})$ be a partition of $\mathbb{N}$ such that $J_{n}\cap A = I_{n}$ for every $n$. Since $\mathcal{L}$ is a non-meager ideal  there exists an infinite $X\subseteq \mathbb{N}$ such that $\bigcup_{n\in X} J_{n}\in \mathcal{L}$. For $M=\bigcup_{n\in X} J_{n}$ we have $\bigcup_{n\in X} I_{n} \subseteq supp(P_{M}a)\notin \mathcal{J}$ and clearly $\|a_{n}\|\geq \epsilon$ for every $n\in \bigcup_{n\in X} I_{n}$. Hence $P_{M}a\notin \mathcal{D}^{\mathcal{J}}[\vec{E}]$.

\end{proof}

Next lemma shows that every locally topologically trivial automorphism in the extension is forced to have a "simple" definition.
\begin{lemma}\label{6}
Assume  $\mathbb{P}=\{\mathbb{P}_{\xi}, \dot{\mathbb{Q}}_{\eta}~ :~ \xi\leq \mathfrak{c}^{+} , \eta< \mathfrak{c}^{+}\}$ is as above where $\dot{\mathbb{Q}}_{0}$ is the poset for the random forcing and assume $\dot{\Phi}$ is a $\mathbb{P}$-name for an automorphism which extends a locally topologically trivial ground model automorphism $\Phi : \mathcal{C}^{\mathcal{J}}[\vec{E}]\rightarrow \mathcal{C}^{\mathcal{J}}[\vec{E}]$ such that $int_{G}{\dot{\Phi}}$ is itself locally topologically trivial with the same local continuous maps witnessing local triviality of $\Phi$, then there exists a $q\in \mathbb{P}$ such that
\begin{equation}
 \nonumber  q\Vdash\{(a,b) : \Phi(\pi_{\mathcal{J}}(a)) =  \pi_{\mathcal{J}}(b)\}  \text{  is  }  \Pi^{1}_{2}.
\end{equation}

\end{lemma}
\begin{proof}
Let $\dot{g}_{\xi}$ be the canonical $\mathbb{Q}_{\xi}$-name for the generic real added by $\mathbb{Q}_{\xi}$ and let $\dot{y}$ be a $\mathbb{P}$-name such that $\dot{\Phi}(\pi_{\mathcal{J}}(\dot{g}_{0}))=\pi_{\mathcal{J}}(\dot{y})$. Note that $\dot{g}_{0}$ is the canonical $\mathbb{Q}_{0}$-name for the random real. Since $\mathbb{P}$ has continuous reading of names, we can find a condition $p$ with countable support $S$ containing $0$, a compact set $K\subset (F^{\vec{I}})^{S}$ and a continuous function $h: K\rightarrow F^{\vec{I}}$ such that $p\Vdash h(\langle \dot{g}_{\xi}:~\xi\in S\rangle)=\dot{y}$.\\
 Let $\mathcal{Z}$ be the set  of all pairs  $(M,N,f)$ such that
 \begin{enumerate}
 \item $M,N\subseteq\mathbb{N}$.
  \item $f : \mathcal{D}_{M}[\vec{E}]\rightarrow \mathcal{D}_{N}[\vec{E}]$ is a continuous representation of a *-homomorphism from  $\mathcal{C}^{\mathcal{J}}_{M}[\vec{E}]$ into $\mathcal{C}^{\mathcal{J}}_{N}[\vec{E}]$.
  \item $f(P_{M})=^{\mathcal{J}}P_{N}$.
  \item $f(a)\in\mathcal{D}^{\mathcal{J}}[\vec{E}]$ if and only if $a\in \mathcal{D}^{\mathcal{J}}[\vec{E}]\cap \mathcal{D}_{M}[\vec{E}]$.
  \item $p\Vdash f(\check{P}_{M}\dot{g}_{0})=^{\mathcal{J}} \check{P}_{N}\dot{y}$.
 \end{enumerate}

 It is not hard to see that conditions (1),(2),(3), and (4) are $\Pi_{1}^{1}$ and therefore by (co)analytic absoluteness still hold in the generic extension.  Moreover by lemma  \ref{222}  condition (5) is $\Delta_{2}^{1}$. Therefore $\mathcal{Z}$ is $\Delta_{2}^{1}$. The set
  \begin{equation}
 \nonumber \Gamma = \{M~:~ (M,N, f)\in\mathcal{Z} \text{ for some $N$ and $f$}\}
 \end{equation}
 is an ideal on $\mathbb{N}$ and $ Triv_{\Phi}^{0}\subseteq\Gamma$. Since $\Phi$ is locally topologically trivial, $\Gamma$ is non-meager.
 For any  $M\in\Gamma$ let $f_{M}$ be such that $(M,N, f_{M})\in\mathcal{Z}$ for some $N\subseteq\mathbb{N}$. Let $\Phi_{*}$ be an arbitrary representation of the extension of $\Phi$ in the forcing extension.

 \textbf{Claim 1:} For all $M\in \Gamma$ we have $f_{M}(a)=^{\mathcal{J}} \Phi_{*}(a)$ for every $a$ in $\mathcal{D}_{M}(\vec{E})$.

This clearly holds for any finite $M$. Assume $M\in \Gamma$ is infinite.
 By our assumption $p$ forces that

  \begin{equation}
   \nonumber f_{M}(P_{M}\dot{g}_{0})=^{\mathcal{J}} P_{N}\Phi_{*}(\dot{g}_{0}).
  \end{equation}

 Now by corollary \ref{2}, since $P_{M}\dot{g}_{0}$ is the random real with respect to $\prod_{n\in M}G_{n}$,  for every $a$ in $\mathcal{D}_{M}(\vec{E})$

  \begin{equation}
  \nonumber f_{M}(a)=^{\mathcal{J}}P_{N} \Phi_{*}(a).
  \end{equation}

  Let $d=(I - P_{N})\Phi_{*}(a)$. It's enough to show that $d\in \mathcal{D}^{\mathcal{J}}[\vec{E}]$. Let $c=\Phi^{-1}_{*}(d)$ and note that  $(I-P_{M})c\in \mathcal{D}^{\mathcal{J}}[\vec{E}]$ since
    \begin{equation}
   \nonumber \Phi_{*}((I - P_{M})c)=^{\mathcal{J}} \Phi_{*}(I-P_{M})\Phi_{*}(a)(I-P_{N}) =^{\mathcal{J}} 0 .
  \end{equation}

  On the other hand we have
   \begin{equation}
   \nonumber f_{M}(P_{M}c)=^{\mathcal{J}}P_{N}\Phi_{*}(P_{M}c)=^{\mathcal{J}} 0.
  \end{equation}

  By assumption (4) we have  $P_{M}c\in\mathcal{D}^{\mathcal{J}}[\vec{E}]$. This implies $c$ and hence $d$ belong to $\mathcal{D}^{\mathcal{J}}[\vec{E}]$.

As a consequence of  claim (1) if $M\in \Gamma$ then $f_{M}$ witnesses that $M\in Triv_{\Phi}^{0}$ and therefore $\Gamma= Triv_{\Phi}^{0}$.

\textbf{Claim 2: } The following holds in the generic extension:
\begin{equation}
\nonumber\{(a,b) ~:~ \Phi(\pi_{\mathcal{J}}(a))= \pi_{\mathcal{J}}(b)\} = \{(a,b)~:~ (\forall (M, N, f)\in\mathcal{Z}) ~f(P_{M} a)=^{\mathcal{J}} P_{N}b\}.
\end{equation}

Suppose $\Phi(\pi_{\mathcal{J}}(a))= \pi_{\mathcal{J}}(b)$. Again let $\Phi_{*}$ be an arbitrary representation of the extension of $\Phi$ in the forcing extension.  For any $(M, N, f)\in\mathcal{Z}$ by claim (1) we have $f(P_{M}a)=^{\mathcal{J}} \Phi_{*}(P_{M}a)=^{\mathcal{J}} P_{N}b$.

 To see the other direction take $(a,b)$ such that $\Phi(\pi_{\mathcal{J}}(a))\neq \pi_{\mathcal{J}}(b)$. Since $\Phi$ is an automorphism we can find a $\mathcal{D}^{\mathcal{J}}[\vec{E}]$-positive element $c$ such that $\Phi_{*}(c)=^{\mathcal{J}} \Phi_{*}(a)-b$. Since $\Gamma$ is a non-meager ideal by lemma \ref{3} we can find an infinite $M\in\Gamma$ such that $P_{M}c$ is  $\mathcal{D}^{\mathcal{J}}[\vec{E}]$-positive. Now for $(M, N, f_{M})\in \mathcal{Z}$ we have
\begin{eqnarray}
\nonumber f_{M}(P_{M}a)-P_{N}b =^{\mathcal{J}} \Phi_{*}(P_{M}a)- \Phi_{*}(P_{M})b =^{\mathcal{J}}\Phi_{*}(P_{M})(\Phi_{*}(a)- b)=^{\mathcal{J}}\Phi_{*}(P_{M}c)
\end{eqnarray}
and therefore $(a,b)$ does not belong to the left hand side of the equation.

This completes the proof since the right hand side of the equation is $\Pi_{2}^{1}$.
\end{proof}
\section{trivial automorphisms}
\textbf{Proof of theorem \ref{1}.}
 Start with a countable model of \textbf{ZFC}+\textbf{MA} and consider the countable support iteration $\mathbb{P}=\{\mathbb{P}_{\xi},\dot{\mathbb{Q}}_{\eta}~:~\xi\leq \mathfrak{c}^{+}, \eta<\mathfrak{c}^{+}\}$ of forcings of the form $\mathbb{S}_{F^{\vec{I}}}$ and the random forcing such that
 \begin{enumerate}
  \item For every partition $\vec{I}$ of $\mathbb{N}$ into finite intervals  the set $\{\xi : \mathbb{Q}_{\xi} \text{ is } \mathbb{S}_{F^{\vec{I}}} \text{ and } cf(\xi)>\aleph_{0} \}$ is stationary.
 \item The set $\{\xi : \mathbb{Q}_{\xi}$  is the random forcing  and  $ cf(\xi)>\aleph_{0}\}$ is also a stationary set.
  \end{enumerate}
  Let $G$ be a generic filter on $\mathbb{P}$.
  Fix a $\mathbb{P}$-name $\dot{\mathcal{J}}$ for a Borel ideal on $\mathbb{N}$ and a $\mathbb{P}$-name  $\dot{\Phi}$ for an automorphism of $\mathcal{C}^{\mathcal{J}}[\vec{E}]$ in the extension. Since every partition is captured in stationary many steps of uncountable cofinalities, by lemma \ref{4}  $\dot{\Phi}$ is forced to be a $\mathbb{P}$-name for a locally $\Delta_{2}^{1}$ automorphism. Each $\mathbb{P}_{\xi}$ is proper, hence no reals are added at stages of uncountable cofinality. For every $\eta$ with uncountable cofinality $H(\aleph_{1})^{V[G\upharpoonright \eta]}$ is the direct limit of $H(\aleph_{1})^{V[G\upharpoonright \xi]}$ for $\xi<\eta$. By a basic model theory fact there is a club $C$ relative to $\{\xi<\mathfrak{c}^{+}: cf(\xi)\geq \aleph_{1}\}$ such that for every $\xi\in C$ and $\dot{A}$ a $\mathbb{P}$-name for a set of reals we have
  \begin{equation}
  \nonumber (H(\aleph_{1}), int_{G\upharpoonright \xi}(\dot{A}\upharpoonright\xi))^{V[G\upharpoonright\xi]}\preceq (H(\aleph_{1}), int_{G}(\dot{A}))^{V[G]}.
  \end{equation}

    Therefore for every $\xi\in\textbf{C}$, $\dot{\Phi}\upharpoonright\xi$ is a $\mathbb{P}_{\xi}$-name for a locally $\Delta_{2}^{1}$ automorphism and $cf(\xi)>\aleph_{0}$. Fix such a $\xi$ and by (2) assume $\dot{\mathbb{Q}}_{\xi}$ is the name for the random forcing. By \textbf{MA} in the ground model and applying lemma \ref{999} locally we can find  Baire-measurable and hence continuous  representations of $\dot{\Phi}$ in $V$. Therefore  $\dot{\Phi}$ is a $\mathbb{P}_{[\xi, \mathfrak{c}^{+}]}$-name for a locally topologically trivial automorphism which its local triviality is witnessed by ground model continuous maps.
    Therefore lemma \ref{6} implies that $int_{G}(\dot{\Phi})$ is forced to be $\Pi_{2}^{1}$ in $V[G]$.
Since our assumption that there is a measurable cardinal implies that $\Pi_{2}^{1}$ sets have $\Pi_{2}^{1}$-uniformizations and all $\Pi_{2}^{1}$ sets have the property of Baire, the automorphism $int_{G}\dot{\Phi}$  has a Baire-measurable and hence a continuous representation. If $\mathcal{J}$ is a Borel P-ideal by lemma $\ref{5}$ we can get a representation of $int_{G}\dot{\Phi}$   which is a *-homomorphism.
\begin{flushright}
$\square$
\end{flushright}

The following corollary is essentially proved in \cite{FaSh} where the authors show the consistency of having all automorphisms of $P(\mathbb{N})/\mathcal{I}$ trivial for a Borel ideal $\mathcal{I}$  while the Calkin algebra has an outer automorphism.
\begin{corollary}\label{654}
It is relatively consistent with \textbf{ZFC} that all automorphisms of $\mathcal{C}^{\mathcal{J}}[\vec{E}]$ are (trivial) topologically trivial for a Borel (P-)ideal $\mathcal{J}$ and every partition $\vec{E}$ of natural numbers into finite intervals while the Calkin algebra has an outer automorphism.
\end{corollary}
\begin{proof}
Since $\mathbb{P}$ is a countable support iteration  of proper $\omega^{\omega}$-bounding
forcings, it is proper and $\omega^{\omega}$-bounding \cite[$\S$ xVI.2.8(D)]{ShProper}. Hence the dominating number $\mathfrak{d}=\aleph_{1}$. This and the weak continuum hypothesis $2^{\aleph_{0}}<2^{\aleph_{1}}$ imply that the Calkin algebra has an outer automorphism (see \cite{FaCalkin}, the paragraph after the proof of Theorem 1.1). In order to get $2^{\aleph_{0}}<2^{\aleph_{1}}$ start with a model of \textbf{CH} and force with the poset consisting of all countable partial functions $f:\aleph_{3}\times \aleph_{1}\rightarrow \{0,1\}$ ordered by the reverse inclusion to add $\aleph_{3}$ so-called Cohen subsets of $\aleph_{1}$. This will increase $2^{\aleph_{1}}$ to $\aleph_{3}$ while preserving \textbf{CH}. Now force with $\mathbb{P}$  the iteration of length $\aleph_{2}$ as above to make all  automorphisms of $\mathcal{C}^{\mathcal{J}}[\vec{E}]$ trivial. A simple $\Delta$-system argument shows that $\mathbb{P}$ is $\aleph_{2}$-cc and hence it preserves $2^{\aleph_{1}}$.
\end{proof}

\section{concluding remarks}

We don't know whether the large cardinal assumption in theorem \ref{1} is necessary. We have partially removed the need for this assumption in our proof, but as it is pointed out in \cite{FaSh} it is likely that one can completely remove it.

The forcing $\mathbb{P}$ used in this article in fact can be written as a countable support iteration of the random forcing and a single groupwise Silver forcing in the way described in the proof of the theorem \ref{1}. To see this notice that if two partitions of natural numbers $\vec{I}$ and $\vec{J}$ are such that $\vec{J}$ is coarser than $\vec{I}$, then $\mathbb{S}_{F^{\vec{J}}}$ captures $F^{\vec{I}}$. Let $\vec{J}=(J_{n})$ be such that $|J_{n}|=n$. It is enough to show that for every $\vec{I}$ there exists a condition $p$ in $\mathbb{S}_{F^{\vec{J}}}$ such that the partial order $\{q\in\mathbb{S}_{F^{\vec{J}}}: q\leq p\}$ is forcing equivalent to $\mathbb{S}_{F^{\vec{I}}}$. By the remark above we can assume $|I_{n}|=k_{n}$ is increasing. Let $p$ be such that $dom(p)=\mathbb{N}\setminus \{k_{1}, k_{2}, \dots\}$ . Clearly any such $p$ is a condition in $\mathbb{S}_{F^{\vec{J}}}$ since $|J_{n}|=n$. Now it is not hard to check that $\{q\in\mathbb{S}_{F^{\vec{J}}}: q\leq p\}$ and $\mathbb{S}_{F^{\vec{I}}}$ are forcing equivalent.

Note that the results of this paper and \cite{FaSh} can not be immediately modified to work for the category of compact metric groups; for example in \textbf{ZFC} the quotient group $\prod \mathbb{Z}/2\mathbb{Z}/\bigoplus \mathbb{Z}/2\mathbb{Z}$ has $2^{\mathfrak{c}}$ automorphisms and therefore it has nontrivial automorphisms, see \cite[Proposition~9]{FaLift}.

We end with the following question.\\ \\
\textbf{Question 7.1.}
Are there Borel ideals $\mathcal{I}$ and $\mathcal{J}$ such that the assertion that '$\prod_{n} \mathbb{M}_{n}(\mathbb{C})/\bigoplus_{\mathcal{I}} \mathbb{M}_{n}(\mathbb{C})$ is isomorphic to $\prod_{n} \mathbb{M}_{n}(\mathbb{C})/\bigoplus_{\mathcal{J}}\mathbb{M}_{n}(\mathbb{C})$' is independent from \textbf{ZFC}?

\end{document}